\documentclass[review,reqno]{siamart0216}

\usepackage[leqno]{amsmath}
\usepackage{graphicx}
\usepackage{subfigure}
 \usepackage{threeparttable}
 \usepackage{algorithm}  
\usepackage{algorithmic}
\newtheorem{remark}[theorem]{Remark}

\nolinenumbers
\usepackage{bm}
\usepackage{amssymb}
\usepackage{cases}
    \title{Stability and Error estimates of  the SAV Fourier-spectral method for the Phase Field Crystal Equation
    \thanks{The work of X. Li is supported by the Postdoctoral Science Foundation of China under grant numbers BX20190187 and 2019M650152. The work of J. Shen is supported in part by NSF grants  DMS-1620262, DMS-1720442 and AFOSR  grant FA9550-16-1-0102.}
}
    \author{ Xiaoli Li
        \thanks{School of Mathematical Sciences and Fujian Provincial Key Laboratory on Mathematical Modeling and High Performance Scientific Computing, Xiamen University, Xiamen, Fujian, 361005, China. Email: xiaolisdu@163.com}.
        \and Jie Shen 
         \thanks{Corresponding Author. Department of Mathematics, Purdue University, West Lafayette, IN 47907, USA. Email: shen7@purdue.edu}.
}

\begin{document}

\maketitle

\begin{abstract}
We consider fully discrete schemes based on the  scalar auxiliary variable (SAV) approach and stabilized SAV approach in time and the Fourier-spectral method in space for the phase field crystal (PFC) equation. Unconditionally energy stability is established for both first- and second-order fully discrete schemes. In addition to the stability,  we also provide a rigorous error estimate which shows that  our second-order in time with  Fourier-spectral method in space converges with order $O(\Delta t^2+N^{-m})$, where $\Delta t$, $N$ and $m$ are time step size, number of Fourier modes in each direction, and regularity index in  space, respectively. We also present numerical experiments  to verify our theoretical results and demonstrate the robustness and accuracy of the schemes.
\end{abstract}

 
\pagestyle{myheadings}
\thispagestyle{plain}
\markboth{XIAOLI LI AND JIE SHEN} {the SAV and Fourier-spectral method  for the PFC Equation}
 \section{Introduction}
 The phase field crystal equation, which was developed in \cite{elder2004modeling,elder2002modeling}, has been frequently used  in the study of the microstructural evolution of crystal growth on atomic length and diffusive time scales. It is well known that the crystal growth is the major stage in crystallization, which is an important step in the purification of solid compounds. The PFC equation is a sixth-order nonlinear parabolic equation and the phase field variable is introduced to describe the phase transition from the liquid phase to the crystal phase. The PFC equation has been employed to simulate a number of physical phenomena, including crystal growth in a supercooled liquid, dendritic and eutectic solidification, epitaxial growth, material hardness and reconstructive phase transitions.  

It is challenging to develop efficient and accurate  numerical schemes for  the PFC equation  due to the six-order spacial derivative and its   nonlinearity. 
Gomez and Nogueira \cite{gomez2012unconditionally} proposed a numerical algorithm for the phase field crystal equation which is second-order time-accurate and unconditionally stable. Local discontinuous Galerkin method has been developed by Guo and Xu \cite{guo2016local} for the the PFC equation, which is based on the first order and second order convex splitting principle. Li and Kim \cite{li2017efficient} studied an efficient and stable compact fourth-order finite difference scheme for the phase field crystal equation. 
It is worth noting that all these schemes are nonlinear, so their implementations are relatively complex and costly compared to linear schemes. To obtain linear schemes for this model, the main difficulty is  how to discretize the quartic potential. Yang and Han \cite{yang2017linearly} constructed linearly unconditionally energy stable schemes for the PFC equation by adopting the "Invariant Energy Quadratization" (IEQ) approach. They established the unconditionally energy stability. We should point out  that although there are many works on the numerical simulation of the PFC model, very few  are with  convergence analysis and error estimates. Note that Wise, Wang and Lowengrub \cite{wise2009energy} proved the error estimates for the nonlinear first order finite difference method based on the convex splitting method. 

The main goals of this paper are  to construct linear and unconditionally energy stable schemes based on the recently proposed scalar auxiliary variable (SAV) approach \cite{shen2018scalar}, and provide rigorous error analysis for them.  The work presented in this paper for the PFC model is unique in the following aspects. First, we construct two linear, unconditional energy stable schemes for the PFC model based on the stabilized scalar auxiliary variable (S-SAV) approach in time and  Fourier-spectral method in space, where  extra stabilized terms are added, compared with \cite{yang2017linearly,wise2009energy}, while keeping the required accuracy. Second, we carry out rigorous error analysis, which is made possible by the uniform bound of the discrete solutions that we  derive thanks to the unconditional  energy stability. We believe that this is the first such result for any fully discrete linear schemes for the PFC model. 

The paper is organized as follows. In Section 2 we introduce the governing system and some preliminaries. In Section 3 we present  the fully discrete schemes using the S-SAV approach in time  and  Fourier-spectral method in space, and we prove their unconditional energy stability In Section 4. In Section 5 we provide rigorous  error estimate for our fully discrete schemes. In Section 6 we present some numerical experiments  to verify the accuracy of the proposed numerical schemes. Some concluding remarks are given in the last section.

\section{Governing system and some preliminaries} \label{sec:Preliminary}
Consider the free energy of Swift-Hohenberg type (cf. \cite{swift1977hydrodynamic,yang2017linearly,li2019efficient})
\begin{equation}\label{e_definition of free energy}
\aligned
E(\phi)=\int_{\Omega}\left(\frac 1 2 \phi(\Delta+\beta)^2\phi+\frac 1 4 \phi^4-\frac{\epsilon}{2}\phi^2\right)d\textbf{x},
\endaligned
\end{equation}
where the phase field variable $\phi$ is the atomic density field and $\beta$ and $\epsilon$ are two positive constants such that $\epsilon<\beta^2$ and $\epsilon\ll1$. We assume $\Omega=(0,L_x)\times 
(0,L_y)$ and $\phi$ is $\Omega$-periodic. Then the PFC equation, which describes the phenomena of crystal growth on the atomic length and diffusive time scales, can be modeled by:
\begin{equation}    \label{e_continuous_model}
\left\{
\begin{array}{l}
\displaystyle \frac{\partial \phi}{\partial t}=M\Delta\mu, \ \ \textbf{x}\in \Omega, t>0,\\
\displaystyle \mu=(\Delta+\beta)^2\phi+\phi^3-\epsilon\phi, \ \ \textbf{x}\in \Omega, t>0,\\
\displaystyle \phi(\textbf{x},0)=\phi_0(\textbf{x}),
\end{array}
\right.
\end{equation}
where $M$ is the mobility function, $\mu=\frac{\delta E(\phi)}{\delta \phi}$ is the chemical potential. We impose the periodic boundary conditions, then it can be easily obtained that the PFC equation \eqref{e_continuous_model} is mass-conservative in the sense that $\frac{d}{dt}\int_{\Omega}\phi d\textbf{x}=0$. Besides, we can derive that the system satisfies the following energy law by using integration by parts:
\begin{equation}\label{e_continuous energy law}
\aligned
\frac{d}{dt}E(\phi)=-\|\sqrt{M}\nabla \mu\|^2\leq 0.
\endaligned
\end{equation} 

Now we give some notations that will be used later. For each $s\geq 0$, Let $(\cdot,\cdot)_{H^s}$ and $\|\cdot\|_{H^s}$ be the $H^s(\Omega)$ inner product and norm, respectively. Note that $H^0(\Omega)=L^2(\Omega)$. In particular, we use $(\cdot,\cdot)$ and $\|\cdot\|$ to denote the $L^2$ inner product. We define Sobolev spaces $L^2_{per}(\Omega)=\{v\in L^2(\Omega)|\rm{\ v \ is \ periodic \ on} \ \Omega\}$ and $H^s_{per}(\Omega)=\{v\in H^s(\Omega)|\rm{\ v \ is \ periodic \ on} \ \Omega\}$. Besides,  let $N_T>0$ be a positive integer and $J=(0,T]$ in this paper. Set
 
$$\Delta t=T/N_T,\ t^n=n\Delta t,\ \ \rm{for} \ n\leq N_T,$$
where $T$ is the final time.
 
 Throughout the paper we use $C$, with or without subscript, to denote a positive
constant, which could have different values at different appearances.
  
\section{Fully discrete schemes by the stabilized SAV Fourier-spectral method}
In this section, we first construct semi-discrete stabilized SAV schemes with the first- and second-order accuracy for the PFC equation, followed by the fully discretization with  Fourier-spectral method in space.
\subsection{The semi-discrete schemes}
To construct SAV schemes with a linear stabilization, we recast the second equation in the PFC model \eqref{e_continuous_model} by 
\begin{equation}\label{e_mu_transform}
\aligned
\mu=(\Delta+\beta)^2\phi+\lambda \phi+F^{\prime}(\phi),
\endaligned
\end{equation} 
where $\lambda$ is a positive constant and $F(\phi)= \frac 1 4 \phi^4-\frac{\epsilon+\lambda}{2}\phi^2$. In the SAV approach, a scalar variable $r(t)=\sqrt{E_1(\phi)}$ is introduced, where $E_1(\phi)=\int_{\Omega}F(\phi)d\textbf{x}+C_0$ and $C_0\geq0$ is chosen to satisfy that $E_1(\phi)>0$. Then the PFC model can be transformed into the following system:
  \begin{subequations}\label{e_model_recast}
    \begin{align}
    &\frac{\partial \phi}{\partial t}=M\Delta \mu,   \label{e_model_recastA}\\
   & \mu=(\Delta+\beta)^2\phi+\lambda \phi+\frac{r(t)}{\sqrt{E_1(\phi)}}F^{\prime}(\phi),   \label{e_model_recastB}\\
    &r_t=\frac{1}{2\sqrt{E_1(\phi)}}\int_{\Omega}F^{\prime}(\phi)\phi_t d\textbf{x}. \label{e_model_recastC}
    \end{align}
  \end{subequations}

\textbf{Scheme \uppercase\expandafter{\romannumeral 1} (first-order accuracy):} Assuming $\phi^n$ and $R^n$ are known,  we update $\phi^{n+1}$ and $R^{n+1}$ by solving 
  \begin{subequations}\label{e_semi first order discrete}
    \begin{align}
    &\phi^{n+1}-\phi^n=M\Delta t\Delta \mu^{n+1},   \label{e_semi first order discreteA}\\
   & \mu^{n+1}=(\Delta+\beta)^2\phi^{n+1}+\lambda \phi^{n+1}-S\Delta(\phi^{n+1}-\phi^n)+\frac{R^{n+1}}{\sqrt{E_1(\phi^n)}}F^{\prime}(\phi^n),   \label{e_semi first order discreteB}\\
    &R^{n+1}-R^n=\frac{1}{2\sqrt{E_1(\phi^n)}}\int_{\Omega}F^{\prime}(\phi^n)(\phi^{n+1}-\phi^n) d\textbf{x}, \label{e_semi first order discreteC}
    \end{align}
  \end{subequations}
where $S>0$ is a stabilizing parameter, which is commonly used in the linear stabilization method for solving phase field model (cf. \cite{shen2010numerical,yang2018efficient_schemes}).

\textbf{Scheme \uppercase\expandafter{\romannumeral 2} (second-order accuracy):} Assuming $\phi^n$, $R^n$ and  $\phi^{n-1}$, $R^{n-1}$ are known, then we update $\phi^{n+1}$ and $R^{n+1}$ by solving 
  \begin{subequations}\label{e_semi second order discrete}
    \begin{align}
    &\phi^{n+1}-\phi^n=M\Delta t\Delta \mu^{n+1/2},   \label{e_semi second order discreteA}\\
   & \mu^{n+1/2}=(\Delta+\beta)^2\phi^{n+1/2}+\lambda \phi^{n+1/2}-S\Delta(\phi^{n+1}-2\phi^n+\phi^{n-1}) \notag \\
   & \ \ \ \ \ \ \ \ \ \ 
   +\frac{R^{n+1/2}}{\sqrt{E_1(\tilde{\phi}^{n+1/2})}}F^{\prime}(\tilde{\phi}^{n+1/2}),   \label{e_semi second order discreteB}\\
    &R^{n+1}-R^n=\frac{1}{2\sqrt{E_1(\tilde{\phi}^{n+1/2}))}}\int_{\Omega}F^{\prime}(\tilde{\phi}^{n+1/2}))(\phi^{n+1}-\phi^n) d\textbf{x}, \label{e_semi second order discreteC}
    \end{align}
  \end{subequations}
 where $\tilde{\phi}^{n+1/2}=(3\phi^n-\phi^{n-1})/2$ and  $S>0$ is a stabilizing parameter. For the case of $n=0$, we can computer $\tilde{\phi}^{1/2}$ by the first order scheme.  

\subsection{The fully discrete schemes with Fourier-spectral method in space}
We first describe  the Fourier-spectral framework. We partition the domain $\Omega=(0,L_x)\times (0,L_y)$ uniformly  with size $h_x=L_x/N_x$, $h_y=L_y/N_y$ where  $N_x$ and $N_y$ are positive integers. The Fourier approximation  space is
 $$S_N=\textrm{span}\{e^{i\xi_kx}e^{i\eta_ly}: -\frac{N_x} {2}\leq k\leq \frac{N_x}{2}-1, -\frac{N_y} {2}\leq l\leq \frac{N_y}{2}-1\},$$
 where $i=\sqrt{-1}$, $\xi_k=2\pi k/L_x$ and $\eta_l=2\pi l/L_y$. Then  any function $u(x,y)\in L^2(\Omega)$ can be approximated by
\begin{equation}\label{e_u approximation}
\aligned
u(x,y)\approx u_N(x,y)=\sum\limits_{k=-\frac{N_x} {2}}^{\frac{N_x}{2}-1}\sum\limits_{l=-\frac{N_y} {2}}^{\frac{N_y}{2}-1}\hat{u}_{k,l}e^{i\xi_kx}e^{i\eta_ly},
\endaligned
\end{equation}  
 where the Fourier coefficients are denoted as 
 $$\hat{u}_{k,l}=<u,e^{i\xi_kx}e^{i\eta_ly}>=
 \frac{1}{|\Omega|}\int_{\Omega}ue^{-i(\xi_kx+\eta_ly)}d\textbf{x}.$$
 
In what follows, we take $N=N_x=N_y$ for simplicity. Then the fully discrete schemes with Fourier spectral method in space based on the mixed formulation can be constructed as follows:
 
 \textbf{Scheme \uppercase\expandafter{\romannumeral 1} (first-order accuracy):} Assuming $\phi_N^n$ and $R^n$ are known, then we update $\phi_N^{n+1}$ and $R^{n+1}$ by solving 
  \begin{subequations}\label{e_fully first order discrete}
    \begin{align}
    &(\phi_N^{n+1}-\phi_N^n,q)+M\Delta t(\nabla \mu_N^{n+1},\nabla q)=0, \ \ \forall 
    \ q\in S_N,   \label{e_fully first order discreteA}\\
   & (\mu_N^{n+1},\Psi)=\left((\Delta+\beta)\phi_N^{n+1}, (\Delta+\beta)\Psi\right)
   +\lambda (\phi_N^{n+1},\Psi)\notag\\  
&+S\left(\nabla(\phi_N^{n+1}-\phi_N^n), \nabla \Psi\right) +\frac{R^{n+1}}{\sqrt{E_1(\phi_N^n)}}(F^{\prime}(\phi_N^n),\Psi), \ \ \forall 
    \ \Psi\in S_N,  \label{e_fully first order discreteB}\\
    &R^{n+1}-R^n=\frac{1}{2\sqrt{E_1(\phi_N^n)}}(F^{\prime}(\phi_N^n),\phi_N^{n+1}-\phi_N^n), \label{e_fully first order discreteC}
    \end{align}
  \end{subequations}
 where $\phi_N^0$ is the $L^2$-orthogonal projection of $\phi_0$, which will be defined late.

  \textbf{Scheme \uppercase\expandafter{\romannumeral 2} (second-order accuracy):} Assuming $\phi_N^n$, $R^n$ and  $\phi_N^{n-1}$, $R^{n-1}$  are known, then we update $\phi_N^{n+1}$ and $R^{n+1}$ by solving 
  \begin{subequations}\label{e_fully second order discrete}
    \begin{align}
    &(\phi_N^{n+1}-\phi_N^n,q)+M\Delta t(\nabla \mu_N^{n+1/2},\nabla q)=0, \ \ \forall 
    \ q\in S_N,   \label{e_fully second order discreteA}\\
   & (\mu_N^{n+1/2},\Psi)=\left((\Delta+\beta)\phi_N^{n+1/2}, (\Delta+\beta)\Psi\right)
  +\lambda (\phi_N^{n+1/2},\Psi)  \label{e_fully second order discreteB}\\ 
&\ \ \ \ \ \ \ +S\left(\nabla(\phi_N^{n+1}-2\phi_N^n+\phi_N^{n-1}), \nabla \Psi\right) \notag\\ 
&\ \ \ \ \ \ \ +\frac{R^{n+1/2}}{\sqrt{E_1(\tilde{\phi}_N^{n+1/2})}}(F^{\prime}(\tilde{\phi}_N^{n+1/2}),\Psi),  \ \  \forall \ \Psi\in S_N, \notag\\ 
    &R^{n+1}-R^n=\frac{1}{2\sqrt{E_1(\tilde{\phi}_N^{n+1/2})}}(F^{\prime}(\tilde{\phi}_N^{n+1/2}),\phi_N^{n+1}-\phi_N^n). \label{e_fully second order discreteC}
    \end{align}
  \end{subequations}
  
\section{Unconditional energy stability}  
 In this section, we prove  the unconditional energy stability for the  first- and second-order fully discrete schemes. Same results can be established for their semi-discrete versions using a similar approach so we omit the detail here.
 
 We define the discrete energy as 
\begin{equation}\label{e_discrete energy}
\aligned
\mathcal{E}(\phi_N^n,R^n)=\frac{1}{2}\|(\Delta+\beta)\phi_N^n\|^2+\frac{\lambda}{2}\|\phi_N^n\|^2+R^2-C_0.
\endaligned
\end{equation} 
 \subsection{The first-order scheme}  
 We have the following results for  the  first-order scheme.
 
\begin{theorem}\label{thm: first order energy stability}
Let $S\ge 0$. The first-order fully discrete scheme \eqref{e_fully first order discrete} is unconditionally energy stable in the sense that the following discrete energy law holds for any $\Delta t$:
\begin{equation}\label{e_first order energy stability1}
\aligned
\mathcal{E}(\phi_N^{n+1},R^{n+1})-\mathcal{E}(\phi_N^n,R^n)+S\|\nabla(\phi_N^{n+1}-\phi_N^n)\|^2\leq& -M\Delta t\|\nabla \mu_N^{n+1}\|^2.
\endaligned
\end{equation}
\end{theorem}

\begin{proof}
Taking $q=\mu_N^{n+1}$ and $\Psi=\phi_N^{n+1}-\phi_N^n$ in \eqref{e_fully first order discrete} and multiplying \eqref{e_fully first order discreteC} with $2R^{n+1}$ lead to
\begin{equation}\label{e_first order energy stability2}
\aligned
(\phi_N^{n+1}-\phi_N^n,\mu_N^{n+1})+M\Delta t\|\nabla \mu_N^{n+1}\|^2=0.
\endaligned
\end{equation}

\begin{equation}\label{e_first order energy stability3}
\aligned
 &(\mu_N^{n+1},\phi_N^{n+1}-\phi_N^n)=\left((\Delta+\beta)\phi_N^{n+1}, (\Delta+\beta)(\phi_N^{n+1}-\phi_N^n)\right)\\
  &\ \ \ \ \ \ +\lambda (\phi_N^{n+1},\phi_N^{n+1}-\phi_N^n) 
+S\|\nabla(\phi_N^{n+1}-\phi_N^n)\|^2 \\
&\ \ \ \ \ \ +\frac{R^{n+1}}{\sqrt{E_1(\phi_N^n)}}(F^{\prime}(\phi_N^n),(\phi_N^{n+1}-\phi_N^n)).
\endaligned
\end{equation}

\begin{equation}\label{e_first order energy stability4}
\aligned
(R^{n+1}-R^n,2R^{n+1})=\frac{R^{n+1}}{\sqrt{E_1(\phi_N^n)}}(F^{\prime}(\phi_N^n),\phi_N^{n+1}-\phi_N^n).
\endaligned
\end{equation}
Noting the identity $$2(a-b,a)= a^2-b^2+(a-b)^2$$
 and combing the above equations, we can obtain
 \begin{equation}\label{e_first order energy stability5}
\aligned
&\frac{1}{2}\|(\Delta+\beta)\phi_N^{n+1}\|^2+\frac{\lambda}{2}\|\phi_N^{n+1}\|^2+(R^{n+1})^2+S\|\nabla(\phi_N^{n+1}-\phi_N^n)\|^2\\
&-\left(\frac{1}{2}\|(\Delta+\beta)\phi_N^{n}\|^2+\frac{\lambda}{2}\|\phi_N^{n}\|^2+(R^{n})^2\right)\leq -M\Delta t\|\nabla \mu_N^{n+1}\|^2,
\endaligned
\end{equation}
which implies the result \eqref{e_first order energy stability1} after we drop some positive terms.
\end{proof}

 \subsection{The second-order scheme}   
  We consider now  the  second-order scheme.
  
\begin{theorem}\label{thm: second order energy stability}
Let $S\ge 0$. The second-order fully discrete scheme \eqref{e_fully second order discrete} is unconditionally energy stable in the sense that  the following discrete energy law holds for any $\Delta t$:
\begin{equation}\label{e_second order energy stability1}
\aligned
\tilde{\mathcal{E}}(\phi_N^{n+1},R^{n+1})-\tilde{\mathcal{E}}(\phi_N^n,R^n)=& -M\Delta t\|\nabla \mu_N^{n+1/2}\|^2,
\endaligned
\end{equation}
where $\tilde{\mathcal{E}}(\phi_N^{n+1},R^{n+1})=\mathcal{E}(\phi_N^{n+1},R^{n+1})+S\|\nabla(\phi_N^{n+1}-\phi_N^n)\|^2$
\end{theorem}

\begin{proof}
Taking $q=\mu_N^{n+1/2}$ and $\Psi=\phi_N^{n+1}-\phi_N^n$ in \eqref{e_fully second order discrete} and multiplying \eqref{e_fully second order discreteC} with $2R^{n+1/2}$ lead to
\begin{equation}\label{e_second order energy stability2}
\aligned
(\phi_N^{n+1}-\phi_N^n,\mu_N^{n+1/2})+M\Delta t\|\nabla \mu_N^{n+1/2}\|^2=0.
\endaligned
\end{equation}

\begin{equation}\label{e_second order energy stability3}
\aligned
 &(\mu_N^{n+1/2},\phi_N^{n+1}-\phi_N^n)=\left((\Delta+\beta)\phi_N^{n+1/2}, (\Delta+\beta)(\phi_N^{n+1}-\phi_N^n)\right)\\
  &\ \ \ \ \ \ +\lambda (\phi_N^{n+1/2},\phi_N^{n+1}-\phi_N^n) 
+S\left(\nabla(\phi_N^{n+1}-2\phi_N^n+\phi_N^{n-1}),\nabla(\phi_N^{n+1}-\phi_N^n)\right) \\
&\ \ \ \ \ \ +\frac{R^{n+1/2}}{\sqrt{E_1(\tilde{\phi}_N^{n+1/2})}}(F^{\prime}(\tilde{\phi}_N^{n+1/2}),(\phi_N^{n+1}-\phi_N^n)).
\endaligned
\end{equation}

\begin{equation}\label{e_second order energy stability4}
\aligned
(R^{n+1}-R^n,2R^{n+1/2})=\frac{R^{n+1/2}}{\sqrt{E_1(\tilde{\phi}_N^{n+1/2})}}(F^{\prime}(\tilde{\phi}_N^{n+1/2}),\phi_N^{n+1}-\phi_N^n).
\endaligned
\end{equation}
Then we can obtain the following equation by combing the above equations:
 \begin{equation}\label{e_second order energy stability5}
\aligned
&\frac{1}{2}\|(\Delta+\beta)\phi_N^{n+1}\|^2+\frac{\lambda}{2}\|\phi_N^{n+1}\|^2+(R^{n+1})^2+S\|\nabla(\phi_N^{n+1}-\phi_N^n)\|^2\\
&-\left(\frac{1}{2}\|(\Delta+\beta)\phi_N^{n}\|^2+\frac{\lambda}{2}\|\phi_N^{n}\|^2+(R^{n})^2+S\|\nabla(\phi_N^{n}-\phi_N^{n-1})\|^2\right)\\
& = -M\Delta t\|\nabla \mu_N^{n+1}\|^2,
\endaligned
\end{equation}
which leads to the desired result.
\end{proof}

\begin{remark}
 We observe that both schemes are unconditionally energy stable for all $S \ge 0$, which implies that even without the stabilization term (i.e., $S=0$), both schemes are also unconditionally energy stable. However, as we shall demonstrate through numerical results later, the stabilization terms are essential to obtain accurate results without using exceedingly small time steps.  
\end{remark}
\section{Error estimates}
In this section, we provide rigorous error estimates for the second-order fully discrete scheme \eqref{e_fully second order discrete}. Since the proofs for the first-order scheme \eqref{e_fully first order discrete} are essentially the same as for the second order scheme,  we skip it for brevity.

 Define the $L^2$-orthogonal projection operator $\Pi_N$: $L^2(\Omega)\rightarrow S_N$ by
 \begin{equation}\label{e_projection_L2}
\aligned
(\Pi_Nu-u,\Psi)=0, \ \ \ \forall \ \Psi\in S_N, \ \ \ u\in L^2(\Omega),
\endaligned
\end{equation}
The following results hold (cf. \cite{shen2011spectral,ramos1991c,ainsworth2017analysis}):
For any $\ 0\leq \mu\leq m$, there exists a constant $C$ such that
 \begin{equation}\label{e_projection_error}
\aligned
\|\Pi_Nu-u\|_{\mu}\leq C\|u\|_mN^{\mu-m}, \ \ \ \forall \ u\in H^m_{per}(\Omega),
\endaligned
\end{equation}
where for $m\ge 1$ and $k=0,\cdots,m-1$,
\begin{equation}
 H^m_{per}(\Omega)=\{u\in H^m(\Omega): u^{(k)}(0,\cdot)=u^{(k)}(L_x,\cdot),\; u^{(k)}(\cdot,0)=u^{(k)}(\cdot,L_y)\}.
\end{equation}
Moreover, the operator $\Pi_N$ commutes with the derivation on $H^1_{per}(\Omega)$, i.e.
 \begin{equation}\label{e_projection_commute}
\aligned
\Pi_N\nabla u=\nabla\Pi_Nu, \ \ \ \forall \ u\in H^1_{per}(\Omega).
\endaligned
\end{equation} 

We first demonstrate that energy stability leads to the $H^2$ boundedness of the discrete solutions.

\begin{lemma}\label{lem: uniform boundedness}
Let  $\phi_N^n$ be the solution of \eqref{e_fully second order discrete}. We have $\|\phi_N^n\|_{H^2}\leq C$ for all $n$ and $N$.
\end{lemma}

\begin{proof}
Recalling the energy stability  \eqref{e_second order energy stability1}, we have 
 \begin{equation}\label{e_uniform boundedness1}
\aligned
\|(\Delta+\beta)\phi_N^n\|\leq C,\ \ \ \|\phi_N^n\|\leq C.
\endaligned
\end{equation}
Then we have $\|\Delta \phi_N^n\|\leq C$.
Using integration by parts and Cauchy inequality yields
 \begin{equation}\label{e_uniform boundedness2}
\aligned
\|\nabla\phi_N^n\|=-(\Delta\phi_N^n,\phi_N^n)\leq \frac{1}{2\zeta}\|\Delta \phi_N^n\|+\frac{\zeta}{2}\|\phi_N^n\|\leq C,
\endaligned
\end{equation}
which implies the desired result.
\end{proof}

 For simplicity, we set
\begin{equation*}
\aligned
& e_{\phi}^n=\phi_N^n-\Pi_N\phi^n+\Pi_N\phi^n-\phi^n=\bar{e}_{\phi}^n+\check{e}_{\phi}^n,\\
& e_{\mu}^n=\mu_N^n-\Pi_N\mu^n+\Pi_N\mu^n-\mu^n=\bar{e}_{\mu}^n+\check{e}_{\mu}^n,\\
& e_{r}^n=R^n-r^n.
\endaligned
\end{equation*} 

\begin{theorem}\label{thm: error estimate}
Let $S \ge 0$. Suppose that $\phi\in L^{\infty}(0,T; H^{3}_{per}(\Omega))\bigcap  L^{\infty}(0,T; H^{m+1}_{per}(\Omega))$, $r\in H^3(0,T)$, $\frac{\partial^2 \phi}{\partial t^2}\in L^{2}(0,T; H^{1}_{per}(\Omega))$ and $\frac{\partial^3 \phi}{\partial t^3}\in L^{2}(0,T; H^{-1}(\Omega))$. If $S>0$, we assume $\frac{\partial^2 \phi}{\partial t^2}\in L^{2}(0,T; H^{3}_{per}(\Omega))$ additionally. Then for the second-order fully discrete scheme \eqref{e_fully second order discrete}, we have

\begin{equation}\label{e_error_result}
\aligned
&\|\phi_N^{k}-\phi^k\|^2+(R^k-r^k)^2\leq C\Delta t^4\int_0^{t^k}(
\|(-\Delta)^{-1/2}\frac{\partial^3 \phi}{\partial t^3}\|^2+\|\frac{\partial^2 \phi}{\partial t^2}\|_{H^1}^2\\
&\ \ \ \ \ \ +\|\frac{\partial^2 \phi}{\partial t^2}\|_{H^3}^2+|\frac{d^3 r}{d t^3}|^2)ds+C\|\phi\|_{L^{\infty}(J;H^{m+1}_{per}(\Omega))}^2N^{-2m},
\endaligned
\end{equation}
where $k\leq N_T$ and the constant $C$ is independent on $N$ and $\Delta t$.
\end{theorem}

\begin{proof}
Subtracting \eqref{e_model_recast} from \eqref{e_fully second order discrete} at $t^{n+1/2}$, we find
\begin{equation}\label{e_error1}
\aligned
(\bar{e}_{\phi}^{n+1}-\bar{e}_{\phi}^n,q)+M\Delta t(\nabla \bar{e}_{\mu}^{n+1/2},\nabla q)
=(Q_1^{n+1/2},q),
\endaligned
\end{equation}

\begin{equation}\label{e_error2}
\aligned
&(\bar{e}_{\mu}^{n+1/2},\Psi)=\left((\Delta+\beta)\bar{e}_{\phi}^{n+1/2}, (\Delta+\beta)\Psi\right)+\lambda (\bar{e}_{\phi}^{n+1/2},\Psi)\\
&\ \ \ \ \ \ +\frac{R^{n+1/2}}{\sqrt{E_1(\tilde{\phi}_N^{n+1/2})}}(F^{\prime}(\tilde{\phi}_N^{n+1/2}),\Psi)-\frac{r^{n+1/2}}{\sqrt{E_1(\phi^{n+1/2})}}(F^{\prime}(\phi^{n+1/2}),\Psi)\\
&\ \ \ \ \ \ +S\left(\nabla(\bar{e}_{\phi}^{n+1}-2\bar{e}_{\phi}^n+\bar{e}_{\phi}^{n-1}), \nabla \Psi\right)+S\left(\nabla(\phi^{n+1}-2\phi^n+\phi^{n-1}), \nabla \Psi\right),
\endaligned
\end{equation}

\begin{equation}\label{e_error3}
\aligned
&e_r^{n+1}-e_r^n=\frac{1}{2\sqrt{E_1(\tilde{\phi}_N^{n+1/2})}}(F^{\prime}(\tilde{\phi}_N^{n+1/2}),\phi_N^{n+1}-\phi_N^n)\\
&\ \ \ \ \ \ -\frac{1}{2\sqrt{E_1(\phi^{n+1/2})}}\int_{\Omega}F^{\prime}(\phi^{n+1/2})\phi_t^{n+1/2}\Delta t d\textbf{x}+Q_2^{n+1/2},
\endaligned
\end{equation}
where the truncation errors are give by the Taylor expansion:
\begin{equation}\label{e_truncation_S1}
\aligned
Q_1^{n+1/2}=&\Delta t\frac{\partial \phi^{n+1/2}}{\partial t}-(\phi^{n+1}-\phi^{n})\\
=&\frac{1}{2}\int_{t^{n}}^{t^{n+1/2}}(t^{n}-s)^2\frac{\partial^3 \phi}{\partial t^3}(s)ds+
\frac{1}{2}\int_{t^{n+1}}^{t^{n+1/2}}(t^{n+1}-s)^2\frac{\partial^3 \phi}{\partial t^3}(s)ds.
\endaligned
\end{equation}

\begin{equation}\label{e_truncation_S2}
\aligned
Q_2^{n+1/2}=&\Delta tr_t^{n+1/2}-(r^{n+1}-r^n)\\
=&\frac{1}{2}\int_{t^{n}}^{t^{n+1/2}}(t^{n}-s)^2\frac{d^3 r}{d t^3}(s)ds+
\frac{1}{2}\int_{t^{n+1}}^{t^{n+1/2}}(t^{n+1}-s)^2\frac{d^3 r}{d t^3}(s)ds.
\endaligned
\end{equation}
Taking $q=\bar{e}_{\mu}^{n+1/2}$ and $\Psi=\bar{e}_{\phi}^{n+1}-\bar{e}_{\phi}^n$ in \eqref{e_error1} and  \eqref{e_error2} respectively, and multiplying \eqref{e_error3} with $2e_r^{n+1/2}$ yield
\begin{equation}\label{e_error4}
\aligned
(\bar{e}_{\phi}^{n+1}-\bar{e}_{\phi}^n,\bar{e}_{\mu}^{n+1/2})+M\Delta t\|\nabla \bar{e}_{\mu}^{n+1/2}\|^2
=(Q_1^{n+1/2},\bar{e}_{\mu}^{n+1/2}).
\endaligned
\end{equation}

\begin{equation}\label{e_error5}
\aligned
&(\bar{e}_{\mu}^{n+1/2},\bar{e}_{\phi}^{n+1}-\bar{e}_{\phi}^n)=\frac{1}{2}(\|(\Delta+\beta)\bar{e}_{\phi}^{n+1}\|^2-\|(\Delta+\beta)\bar{e}_{\phi}^n\|^2)+\frac{\lambda}{2}(\|\bar{e}_{\phi}^{n+1}\|^2-\|\bar{e}_{\phi}^n\|^2)\\
&\ \ \  +\frac{R^{n+1/2}}{\sqrt{E_1(\tilde{\phi}_N^{n+1/2})}}(F^{\prime}(\tilde{\phi}_N^{n+1/2}),\bar{e}_{\phi}^{n+1}-\bar{e}_{\phi}^n)-\frac{r^{n+1/2}}{\sqrt{E_1(\phi^{n+1/2})}}(F^{\prime}(\phi^{n+1/2}),\bar{e}_{\phi}^{n+1}-\bar{e}_{\phi}^n)\\
&\ \ \ +S\left(\nabla(\bar{e}_{\phi}^{n+1}-2\bar{e}_{\phi}^n+\bar{e}_{\phi}^{n-1}), \nabla (\bar{e}_{\phi}^{n+1}-\bar{e}_{\phi}^n)\right)\\
&\ \ \ +S\left(\nabla(\phi^{n+1}-2\phi^n+\phi^{n-1}), \nabla (\bar{e}_{\phi}^{n+1}-\bar{e}_{\phi}^n)\right).
\endaligned
\end{equation}

\begin{equation}\label{e_error6}
\aligned
&(e_r^{n+1})^2-(e_r^n)^2=\frac{e_r^{n+1/2}}{\sqrt{E_1(\tilde{\phi}_N^{n+1/2})}}(F^{\prime}(\tilde{\phi}_N^{n+1/2}),\phi_N^{n+1}-\phi_N^n)\\
&\ \ \ \ \ \ -\frac{e_r^{n+1/2}}{\sqrt{E_1(\phi^{n+1/2})}}\int_{\Omega}F^{\prime}(\phi^{n+1/2})\phi_t^{n+1/2}\Delta t d\textbf{x}\\
&\ \ \ \ \ \ +2(Q_2^{n+1/2},e_r^{n+1/2}).
\endaligned
\end{equation}

The term on the right-hand side of \eqref{e_error4} can be estimated by
\begin{equation}\label{e_error7}
\aligned
(Q_1^{n+1/2},\bar{e}_{\mu}^{n+1/2})\leq &\frac{M\Delta t}{6}\|\nabla\bar{e}_{\mu}^{n+1/2}\|^2
+\frac{C}{\Delta t}\|(-\Delta)^{-1/2}Q_1^{n+1/2}\|^2\\
\leq &\frac{M\Delta t}{6}\|\nabla\bar{e}_{\mu}^{n+1/2}\|^2
+C\Delta t^4\int_{t^n}^{t^{n+1}}\|(-\Delta)^{-1/2}\frac{\partial^3 \phi}{\partial t^3}(s)\|^2ds,
\endaligned
\end{equation}
where taking notice of $(Q_1^{n+1/2},1)=0$, the operator $(-\Delta)^{-1/2}$, which is the 
power of $-\Delta$, can be well defined by the spectral theory of self-adjoint operators.

The third and fourth terms on the right-hand side of \eqref{e_error5} can be transformed into
\begin{equation}\label{e_error8}
\aligned
&\frac{R^{n+1/2}}{\sqrt{E_1(\tilde{\phi}_N^{n+1/2})}}(F^{\prime}(\tilde{\phi}_N^{n+1/2}),\bar{e}_{\phi}^{n+1}-\bar{e}_{\phi}^n)-\frac{r^{n+1/2}}{\sqrt{E_1(\phi^{n+1/2})}}(F^{\prime}(\phi^{n+1/2}),\bar{e}_{\phi}^{n+1}-\bar{e}_{\phi}^n)\\
=&\frac{e_r^{n+1/2}}{\sqrt{E_1(\tilde{\phi}_N^{n+1/2})}}(F^{\prime}(\tilde{\phi}_N^{n+1/2}),\bar{e}_{\phi}^{n+1}-\bar{e}_{\phi}^n)+r^{n+1/2}(\frac{F^{\prime}(\tilde{\phi}_N^{n+1/2})}{\sqrt{E_1(\tilde{\phi}_N^{n+1/2})}},\bar{e}_{\phi}^{n+1}-\bar{e}_{\phi}^n)
)\\
&-r^{n+1/2}(\frac{F^{\prime}(\phi^{n+1/2})}{\sqrt{E_1(\phi^{n+1/2})}},\bar{e}_{\phi}^{n+1}-\bar{e}_{\phi}^n).
\endaligned
\end{equation}

Assuming $F(\phi)\in C^3(\mathbb{R})$ and noting \eqref{e_projection_error} and \eqref{e_projection_commute}, the last two terms on the right-hand side of \eqref{e_error8} can be controlled, similar to the estimates in \cite{shen2018convergence}, by
\begin{equation}\label{e_error9}
\aligned
&r^{n+1/2}(\frac{F^{\prime}(\tilde{\phi}_N^{n+1/2})}{\sqrt{E_1(\tilde{\phi}_N^{n+1/2})}},\bar{e}_{\phi}^{n+1}-\bar{e}_{\phi}^n))
-r^{n+1/2}(\frac{F^{\prime}(\phi^{n+1/2})}{\sqrt{E_1(\phi^{n+1/2})}},\bar{e}_{\phi}^{n+1}-\bar{e}_{\phi}^n)\\
= &r^{n+1/2}M\Delta t(\frac{F^{\prime}(\tilde{\phi}_N^{n+1/2})}{\sqrt{E_1(\tilde{\phi}_N^{n+1/2})}}-
\frac{F^{\prime}(\phi^{n+1/2})}{\sqrt{E_1(\phi^{n+1/2})}},\Delta \bar{e}_{\mu}^{n+1/2})\\
&+r^{n+1/2}(\frac{F^{\prime}(\tilde{\phi}_N^{n+1/2})}{\sqrt{E_1(\tilde{\phi}_N^{n+1/2})}}-
\frac{F^{\prime}(\phi^{n+1/2})}{\sqrt{E_1(\phi^{n+1/2})}},Q_1^{n+1/2})\\
\leq &\frac{M\Delta t}{6}\|\nabla\bar{e}_{\mu}^{n+1/2}\|^2+C\Delta t\|\frac{\nabla F^{\prime}(\tilde{\phi}_N^{n+1/2})}{\sqrt{E_1(\tilde{\phi}_N^{n+1/2})}}-
\frac{\nabla F^{\prime}(\phi^{n+1/2})}{\sqrt{E_1(\phi^{n+1/2})}}\|^2\\
&+\frac{C}{\Delta t}\|(-\Delta)^{-1/2}Q_1^{n+1/2}\|^2\\
\leq &\frac{M\Delta t}{6}\|\nabla\bar{e}_{\mu}^{n+1/2}\|^2+C\Delta t(\|\bar{e}_{\phi}^{n}\|^2
+\|\bar{e}_{\phi}^{n-1}\|^2+\|\nabla\bar{e}_{\phi}^{n}\|^2+\|\nabla\bar{e}_{\phi}^{n-1}\|^2)\\
&+C\Delta t^4\int_{t^n}^{t^{n+1}}\|(-\Delta)^{-1/2}\frac{\partial^3 \phi}{\partial t^3}(s)\|^2ds+C\Delta t^4\int_{t^n}^{t^{n+1}}\|\frac{\partial^2 \phi}{\partial t^2}(s)\|_{H^1}^2ds\\
&+C\|\tilde{\phi}^{n+1/2}\|_{m+1}^2N^{-2m}\Delta t,
\endaligned
\end{equation}
where the last inequality holds by the fact that one can find a constant $C$ such that $|F^{\prime}(\phi_N^n)|\le C$, $|F^{\prime\prime}(\phi_N^n)|\leq C$ by using Lemma \ref{lem: uniform boundedness} and the Sobolev embedding theorem $H^2\subseteq L^{\infty}$.

We also have the following using the integration by parts and Cauchy-Schwartz inequality:
\begin{equation}\label{e_error_nabla_estimate}
\aligned
&\|\nabla\bar{e}_{\phi}^{n}\|^2=-(\bar{e}_{\phi}^{n},\Delta \bar{e}_{\phi}^{n})
\leq C(\|\bar{e}_{\phi}^{n}\|^2+\|\Delta\bar{e}_{\phi}^{n}\|^2)
\leq C(\|\bar{e}_{\phi}^{n}\|^2+\|(\Delta+\beta)\bar{e}_{\phi}^{n}\|^2).
\endaligned
\end{equation}

Similar to the estimate in \eqref{e_error9}, the last term on the right-hand side of \eqref{e_error5} can be directly controlled by using 
integration by parts and Cauchy-Schwarz inequality:
\begin{equation}\label{e_error10}
\aligned
&S\left(\nabla(\phi^{n+1}-2\phi^n+\phi^{n-1}), \nabla (\bar{e}_{\phi}^{n+1}-\bar{e}_{\phi}^n)\right)\\
=&-S\left(\Delta(\phi^{n+1}-2\phi^n+\phi^{n-1}), M\Delta t\Delta \bar{e}_{\mu}^{n+1/2}+Q_1^{n+1/2}\right)\\
\leq&\frac{M\Delta t}{6}\|\nabla\bar{e}_{\mu}^{n+1/2}\|^2+C\Delta t^4\int_{t^n}^{t^{n+1}}\|\frac{\partial^2 \phi}{\partial t^2}(s)\|_{H^3}^2ds\\
&+C\Delta t^4\int_{t^n}^{t^{n+1}}\|(-\Delta)^{-1/2}\frac{\partial^3 \phi}{\partial t^3}(s)\|^2ds.
\endaligned
\end{equation}
The first two terms on the right-hand side of \eqref{e_error6} can be recast as
\begin{equation}\label{e_error11}
\aligned
&\frac{e_r^{n+1/2}}{\sqrt{E_1(\tilde{\phi}_N^{n+1/2})}}(F^{\prime}(\tilde{\phi}_N^{n+1/2}),\phi_N^{n+1}-\phi_N^n)-\frac{e_r^{n+1/2}}{\sqrt{E_1(\phi^{n+1/2})}}\int_{\Omega}F^{\prime}(\phi^{n+1/2})\phi_t^{n+1/2}\Delta t d\textbf{x}\\
=&\frac{e_r^{n+1/2}}{\sqrt{E_1(\tilde{\phi}_N^{n+1/2})}}(F^{\prime}(\tilde{\phi}_N^{n+1/2}),\bar{e}_{\phi}^{n+1}-\bar{e}_{\phi}^n)+\frac{e_r^{n+1/2}}{\sqrt{E_1(\tilde{\phi}_N^{n+1/2})}}(F^{\prime}(\tilde{\phi}_N^{n+1/2}),\breve{e}_{\phi}^{n+1}-\breve{e}_{\phi}^n)\\
&+e_r^{n+1/2}(\frac{F^{\prime}(\tilde{\phi}_N^{n+1/2})}{\sqrt{E_1(\tilde{\phi}_N^{n+1/2})}}-
\frac{F^{\prime}(\phi^{n+1/2})}{\sqrt{E_1(\phi^{n+1/2})}},\phi^{n+1}-\phi^n)\\
&-e_r^{n+1/2}(\frac{F^{\prime}(\phi^{n+1/2})}{\sqrt{E_1(\phi^{n+1/2})}},Q_1^{n+1/2}).
\endaligned
\end{equation}
The last two terms on the right-hand side of \eqref{e_error11} can be handled in a similar way as \eqref{e_error9}:
\begin{equation}\label{e_error12}
\aligned
&e_r^{n+1/2}(\frac{F^{\prime}(\tilde{\phi}_N^{n+1/2})}{\sqrt{E_1(\tilde{\phi}_N^{n+1/2})}}-
\frac{F^{\prime}(\phi^{n+1/2})}{\sqrt{E_1(\phi^{n+1/2})}},\phi^{n+1}-\phi^n)\\
\leq &C\Delta t\|\phi_t\|_{L^{\infty}(J;H^{-1})}\left((e_r^{n+1})^2+(e_r^{n})^2\right.+\|\bar{e}_{\phi}^{n}\|^2
+\|\bar{e}_{\phi}^{n-1}\|^2+\|\nabla\bar{e}_{\phi}^{n}\|^2+\|\nabla\bar{e}_{\phi}^{n-1}\|^2)\\
&+C\Delta t^4\int_{t^n}^{t^{n+1}}\|\frac{\partial^2 \phi}{\partial t^2}(s)\|_{H^1}^2ds+C\|\tilde{\phi}^{n+1/2}\|_{H^{m+1}}^2N^{-2m}\Delta t.
\endaligned
\end{equation}

\begin{equation}\label{e_error13}
\aligned
&-e_r^{n+1/2}(\frac{F^{\prime}(\phi^{n+1/2})}{\sqrt{E_1(\phi^{n+1/2})}},Q_1^{n+1/2})
\leq C\Delta t(e_r^{n+1/2})^2\\
& \ \ \ \ \ +C\Delta t^4\int_{t^n}^{t^{n+1}}\|(-\Delta)^{-1/2}\frac{\partial^3 \phi}{\partial t^3}(s)\|^2ds.
\endaligned
\end{equation}
The last term on the right-hand side of \eqref{e_error6} can be controlled by
\begin{equation}\label{e_error6_add}
\aligned
2(Q_2^{n+1/2},e_r^{n+1/2})\leq C\Delta t(e_r^{n+1/2})^2+C\Delta t^4
\int_{t^{n}}^{t^{n+1}}|\frac{d^3 r}{d t^3}(s)|^2ds.
\endaligned
\end{equation}

By combining the above equations, we can obtain 
\begin{equation*}\label{e_error14}
\aligned
&M\Delta t\|\nabla \bar{e}_{\mu}^{n+1/2}\|^2+\frac{1}{2}(\|(\Delta+\beta)\bar{e}_{\phi}^{n+1}\|^2-\|(\Delta+\beta)\bar{e}_{\phi}^n\|^2)+\frac{\lambda}{2}(\|\bar{e}_{\phi}^{n+1}\|^2-\|\bar{e}_{\phi}^n\|^2)\\
&+S(\|\nabla(\bar{e}_{\phi}^{n+1}-\bar{e}_{\phi}^{n})\|^2-\|\nabla(\bar{e}_{\phi}^{n}-\bar{e}_{\phi}^{n-1})\|^2)+(e_r^{n+1})^2-(e_r^n)^2\\
\leq&\frac{M\Delta t}{2}\|\nabla\bar{e}_{\mu}^{n+1/2}\|^2
+C\Delta t(\|\bar{e}_{\phi}^{n}\|^2
+\|\bar{e}_{\phi}^{n-1}\|^2+\|(\Delta+\beta)\bar{e}_{\phi}^{n}\|^2+\|(\Delta+\beta)\bar{e}_{\phi}^{n-1}\|^2)\\
&+C\Delta t\|\phi_t\|_{L^{\infty}(J;H^{-1})}\left((e_r^{n+1})^2+(e_r^{n})^2\right)+C\|\tilde{\phi}^{n+1/2}\|_{m+1}^2N^{-2m}\Delta t\\
&+C\Delta t^4\int_{t^n}^{t^{n+1}}\|(-\Delta)^{-1/2}\frac{\partial^3 \phi}{\partial t^3}(s)\|^2ds+C\Delta t^4\int_{t^n}^{t^{n+1}}\|\frac{\partial^2 \phi}{\partial t^2}(s)\|_{H^1}^2ds\\
&+C\Delta t^4\int_{t^n}^{t^{n+1}}\|\frac{\partial^2 \phi}{\partial t^2}(s)\|_{H^3}^2ds+C\Delta t^4\int_{t^{n}}^{t^{n+1}}(\frac{d^3 r}{d t^3}(s))^2ds
\endaligned
\end{equation*}
Then 
summing over $n$, $n=0,1,\ldots,k-1$, and using Gronwall's inequality, we have
\begin{equation}\label{e_error15}
\aligned
&\sum\limits_{n=0}^{k-1}\Delta t\|\nabla \bar{e}_{\mu}^{n+1/2}\|^2+\|(\Delta+\beta)\bar{e}_{\phi}^{k}\|^2+\|\bar{e}_{\phi}^{k}\|^2+(e_r^{k})^2\\
\leq&C\Delta t^4\int_0^{t^k}(
\|(-\Delta)^{-1/2}\frac{\partial^3 \phi}{\partial t^3}\|^2+\|\frac{\partial^2 \phi}{\partial t^2}\|_{H^1}^2+\|\frac{\partial^2 \phi}{\partial t^2}\|_{H^3}^2+|\frac{d^3 r}{d t^3}|^2)ds\\
&+C\|\phi\|_{L^{\infty}(J;H^{m+1}_{per}(\Omega))}^2N^{-2m},
\endaligned
\end{equation}
By the triangle inequality and \eqref{e_projection_error}, we obtain the desired results \eqref{e_error_result}.
\end{proof}

\begin{remark}
 We observe that with $S>0$, i.e., with an active  stabilization term, we  need to assume additionally 
  $\frac{\partial^2 \phi}{\partial t^2}\in L^{2}(0,T; H^{3}_{per}(\Omega))$. But this additional  regularity requirement might be formally derived from the assumption   $\frac{\partial^3 \phi}{\partial t^3}\in L^{2}(0,T; H^{-1}(\Omega))$ since one time derivative is formally equivalent to six spatial derivatives. 
\end{remark}
\section{Numerical results}
We present in this section several numerical examples to verify the accuracy of our S-SAV Fourier-spectral schemes and to illustrate the phase transition behaviors and crystal growth in a supercooled liquid. In the following simulations, we take $M=1$, $\beta=1$.

\subsection{Accuracy tests} 
In this subsection, we test the accuracy of the first and second order fully discrete schemes \eqref{e_fully first order discrete} and \eqref{e_fully second order discrete}, respectively. In this test, we take $\Omega=(0,32)\times (0,32)$, $T=1$, $S=5$, $\lambda=0.01$, $\epsilon=0.025$ and the initial condition $\phi_0=\sin(\pi x/16)\cos(\pi y/16)$. We measure Cauchy error since we do not have possession of exact solution. Specifically, the error between two different grid spacings $N$ and $2N$ is calculated by $\|e_{\phi}\|=\|\phi_{N}-\phi_{2N}\|$ and we use the similar procedure to compute the error between two different time steps $\Delta t$ and $\Delta t/2$. 

We first check the time accuracy by taking  $N=256$ so the spatial error is negligible. In Tables \ref{table1_example1} and \ref{table2_example1}, we  show the order of convergence in the temporal direction.  To test the spectral accuracy, we choose the time step to be sufficiently small so that the error is dominated by the spatial error, and compute the error at $T=1$, the results are plotted in \ref{fig: spectral accuracy} which shows that the error converges exponentially. These numerical results  are consistent with the error estimates in Theorem \ref{thm: error estimate}.

\begin{table}[htbp]
\renewcommand{\arraystretch}{1.1}
\small
\centering
\caption{Errors and convergence rates in time for the first order scheme  \eqref{e_fully first order discrete}.}\label{table1_example1}
\begin{tabular}{p{1.2cm}p{2.5cm}p{1.2cm}p{1.8cm}p{1.0cm}}\hline
$\Delta t$    &$\|e_{\phi}\|_{L^{\infty}(J;L^2(\Omega))}$    &Rate &$\|e_r\|_{L^{\infty}(J)}$   &Rate   \\ \hline
$1/5$    &1.26E-1                & ---    &9.08E-2         &----\\
$1/10$    &6.77E-2                &0.90    &4.64E-2         &0.97 \\
$1/20$    &3.60E-2                &0.91     &2.33E-2         &0.99 \\
$1/40$    &1.88E-2                &0.94    &1.17E-2         &1.00   \\
$1/80$    &9.67E-3                &0.96    &5.84E-3         &1.00   \\
\hline
\end{tabular}
\end{table}

\begin{table}[htbp]
\renewcommand{\arraystretch}{1.1}
\small
\centering
\caption{Errors and convergence rates in time for the second order scheme  \eqref{e_fully second order discrete}.}\label{table2_example1}
\begin{tabular}{p{1.2cm}p{2.5cm}p{1.2cm}p{1.8cm}p{1.0cm}}\hline
$\Delta t$    &$\|e_{\phi}\|_{L^{\infty}(J;L^2(\Omega))}$    &Rate &$\|e_r\|_{L^{\infty}(J)}$   &Rate   \\ \hline
$1/5$    &4.06E-2                & ---    &7.02E-3         &---\\
$1/10$    &1.02E-2                &1.99    &1.79E-3         &1.97 \\
$1/20$    &2.12E-3                &2.26     &4.41E-4         &2.02 \\
$1/40$    &4.99E-4                &2.09    &1.11E-4         &1.99   \\
$1/80$    &1.22E-4                &2.03    &2.78E-5         &1.99   \\
\hline
\end{tabular}
\end{table}

\begin{figure}[!htp]
\centering
\includegraphics[scale=0.40]{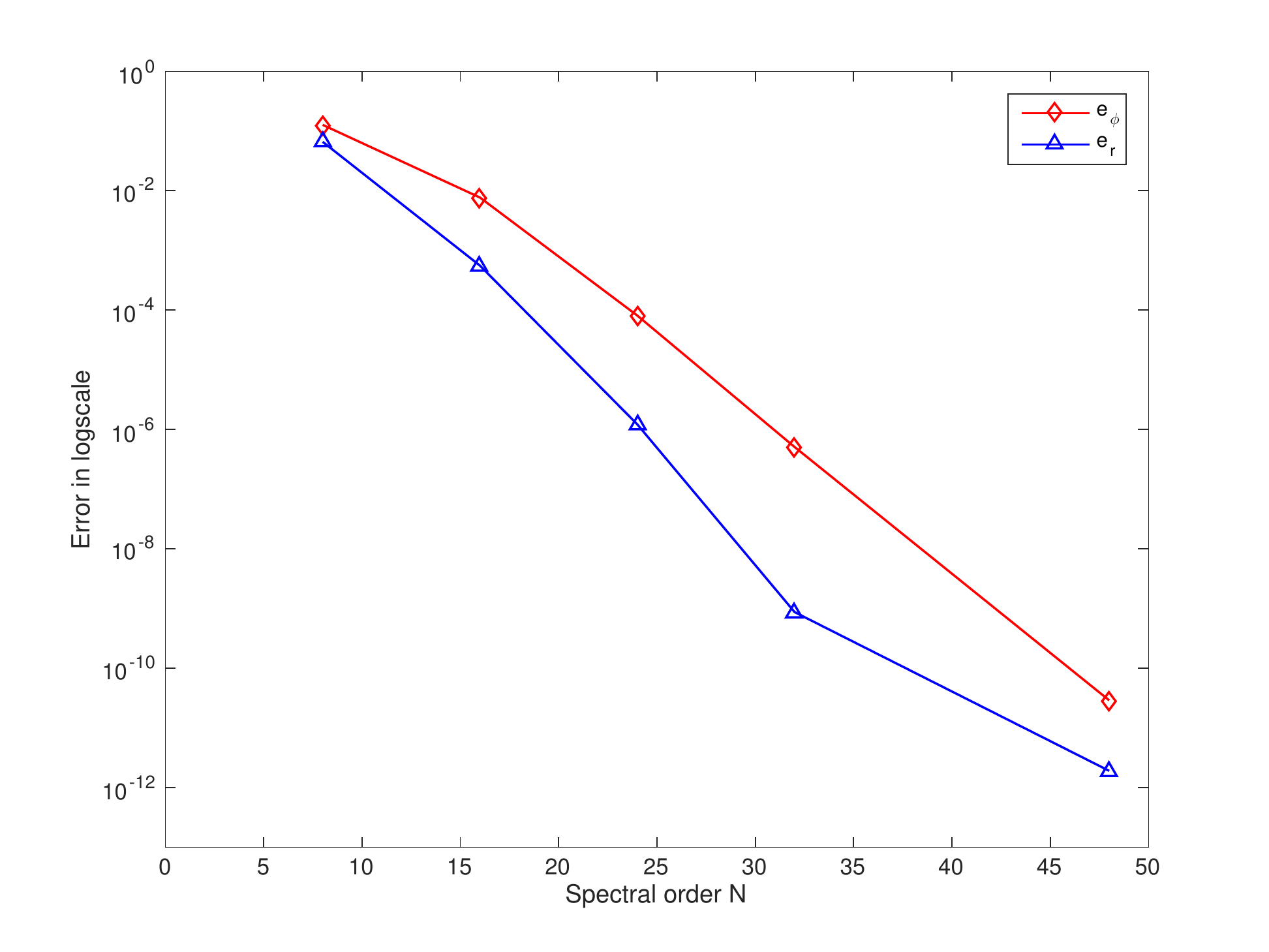}
\caption{Spatial $L^2$ errors at time $T=1$ for the second order scheme \eqref{e_fully second order discrete}}  \label{fig: spectral accuracy}
\end{figure}

\subsection{Phase evolution behaviors and the effect of stabilization} 
In this subsection, we simulate the phase evolution behavior of the PFC model. The physical parameters are set as $\Omega=(0,128)\times (0,128)$ with the random initial data $\phi_{i,j}=\phi_0+\eta_{i,j}$, where $\phi_0=0.06$ and $\eta_{i,j}$ is a uniformly distributed random number satisfying $|\eta_{i,j}|\leq 0.01$. The other parameters are $\epsilon=0.025$,  $\lambda=0.001$.

 In Figure \ref{fig: energy decay}, we present the time evolution of the energy  with stabilization $S=0.01$ and without stabilization. It can be observed that the energy decreases at all times, which indicates numerical evidence for our method being unconditionally energy stable. But the modified SAV energy at large $\Delta t$ is not consistent with the original energy without stabilization, it only becomes close to the original energy with $\Delta t=0.02$. However, with the stabilization parameter $S=0.01$, even the result with $\Delta t=1$ produces reasonably accurate results. This example shows that while the stabilization is not needed for stability, it is essential for accuracy at larger time steps.
 So in the following simulations, we always add a stabilized term so that reasonable accuracy can be achieved  without using exceedingly small time steps.

\begin{figure}[!htp]
\centering
\includegraphics[scale=0.40]{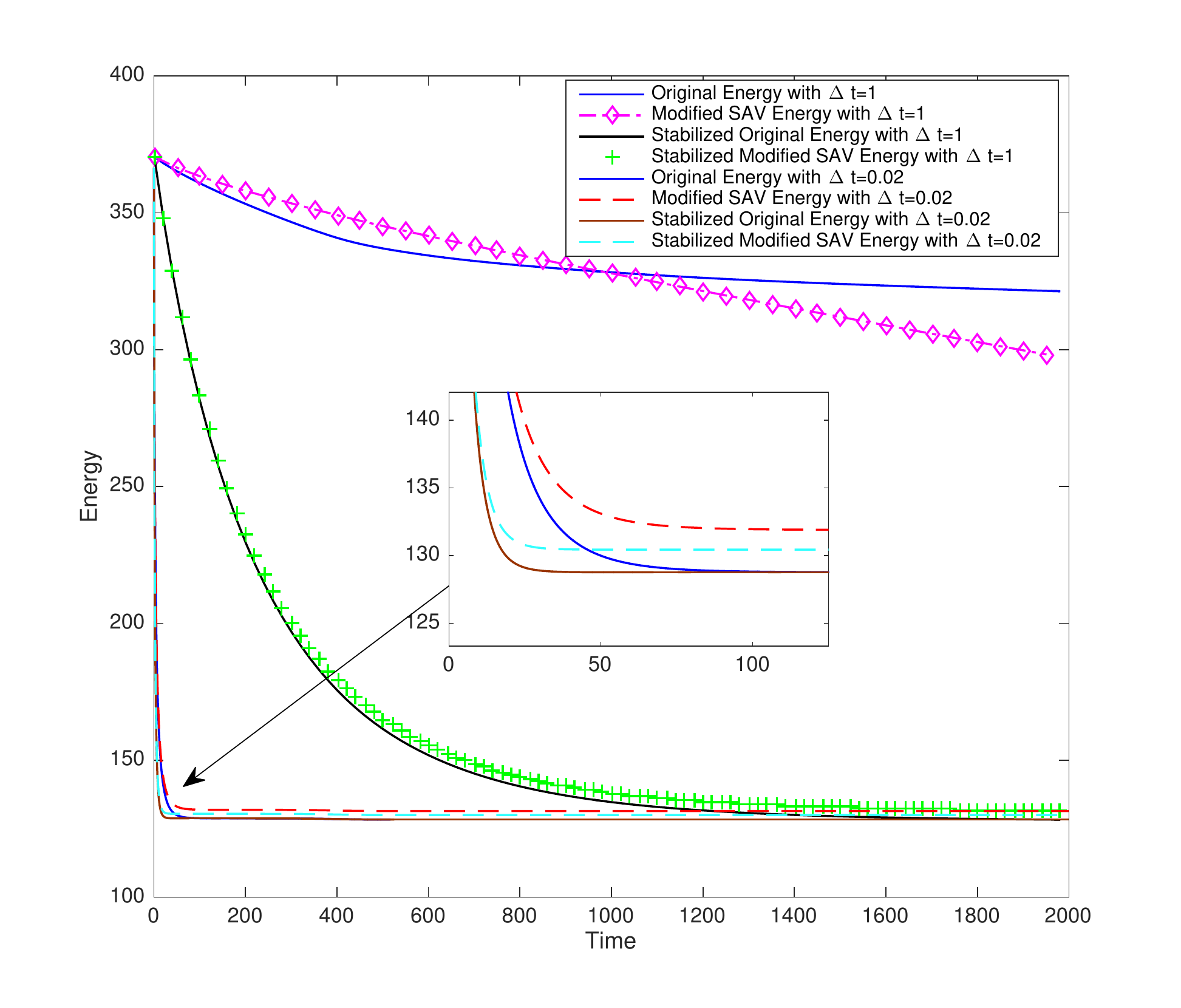}
\caption{Time evolution of the free energy functional.}  \label{fig: energy decay}
\end{figure}

We present the evolution of the density field $\phi$ calculated using the second order scheme \eqref{e_fully second order discrete} with $\Delta t=1$, $S=0.01$ and $N=256$ in Figure \ref{fig: long time simulation}. 
\begin{figure}[htbp]
\centering
\includegraphics[scale=0.21]{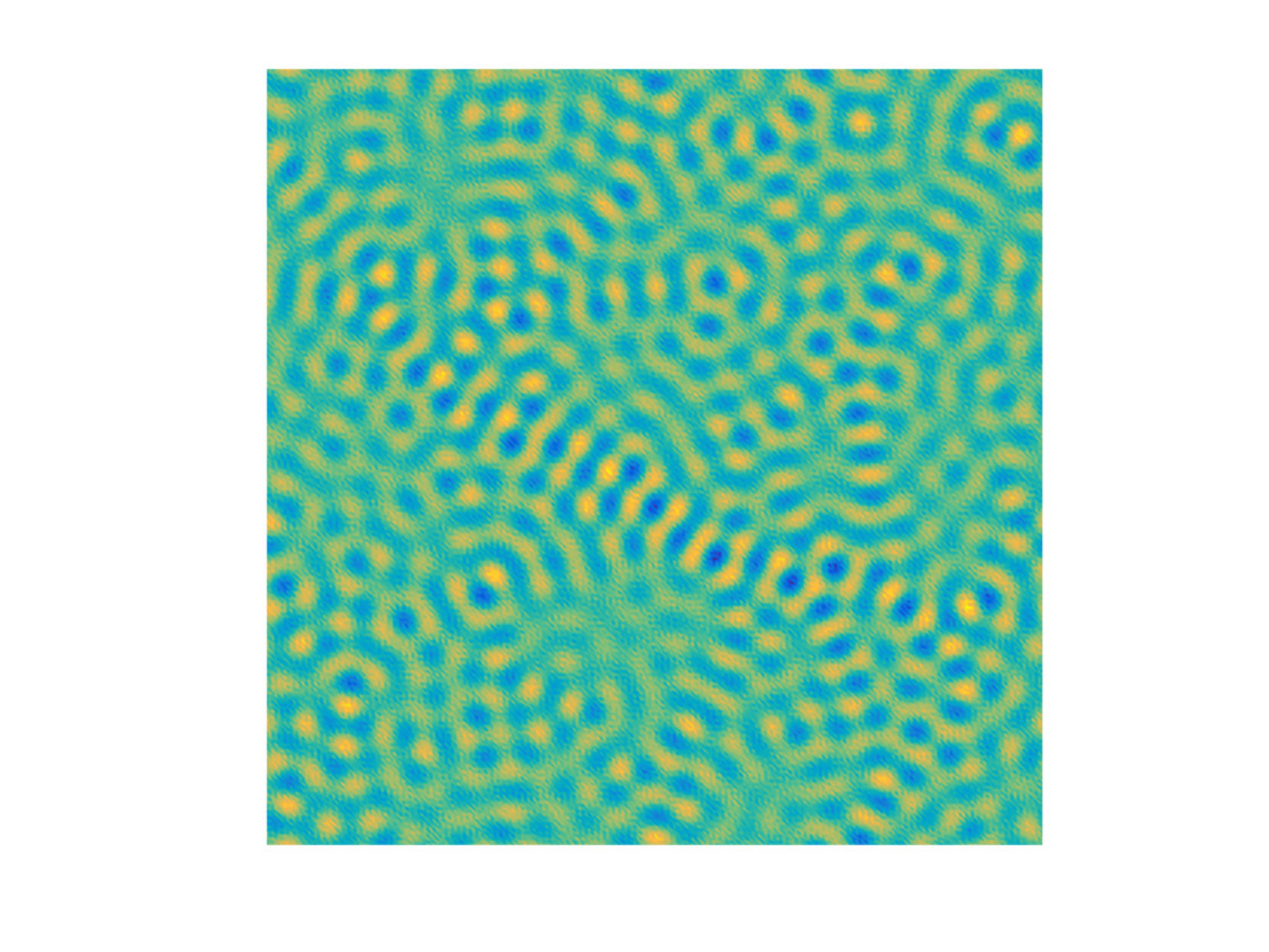}
\includegraphics[scale=0.21]{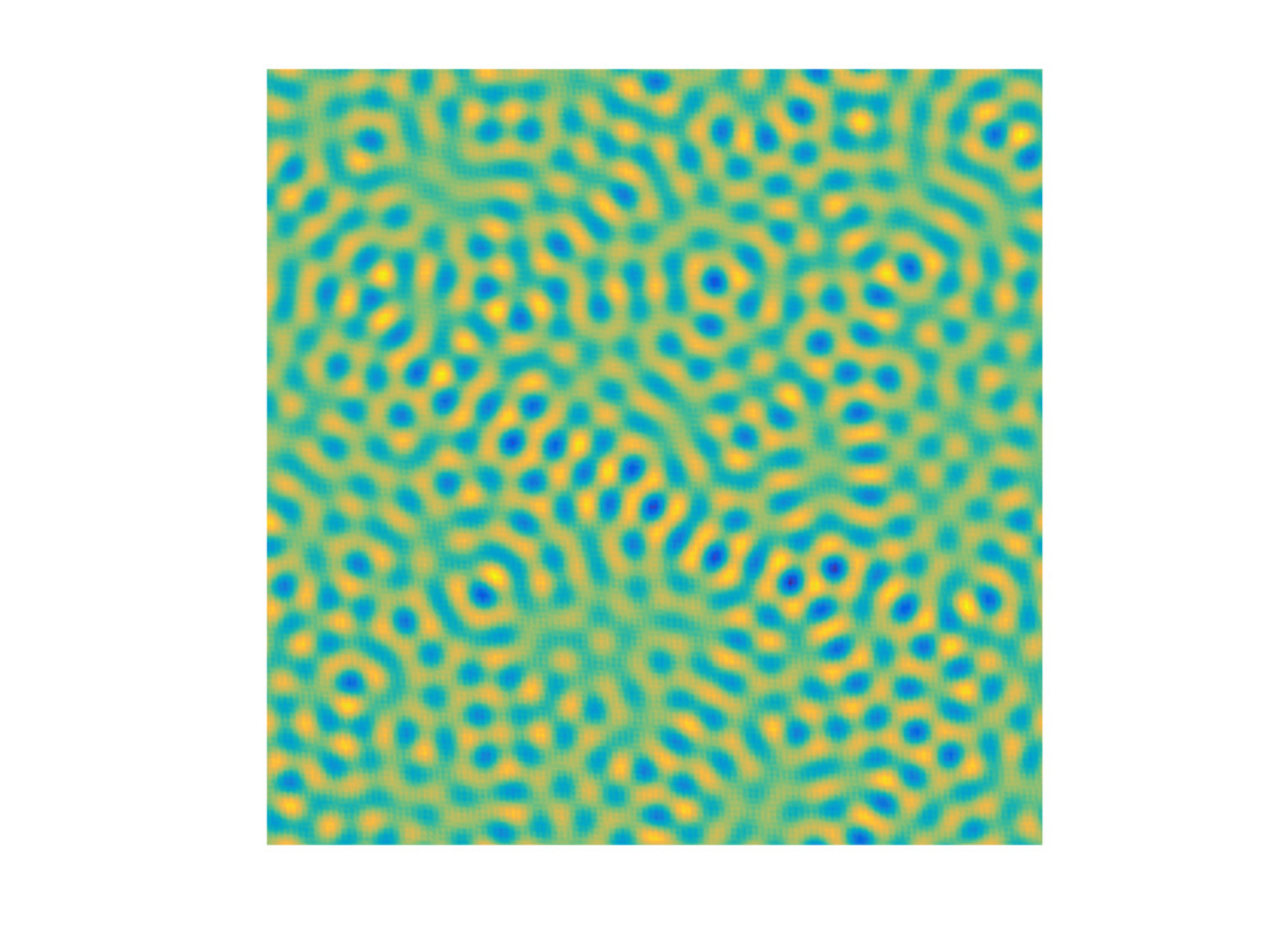}
\includegraphics[scale=0.21]{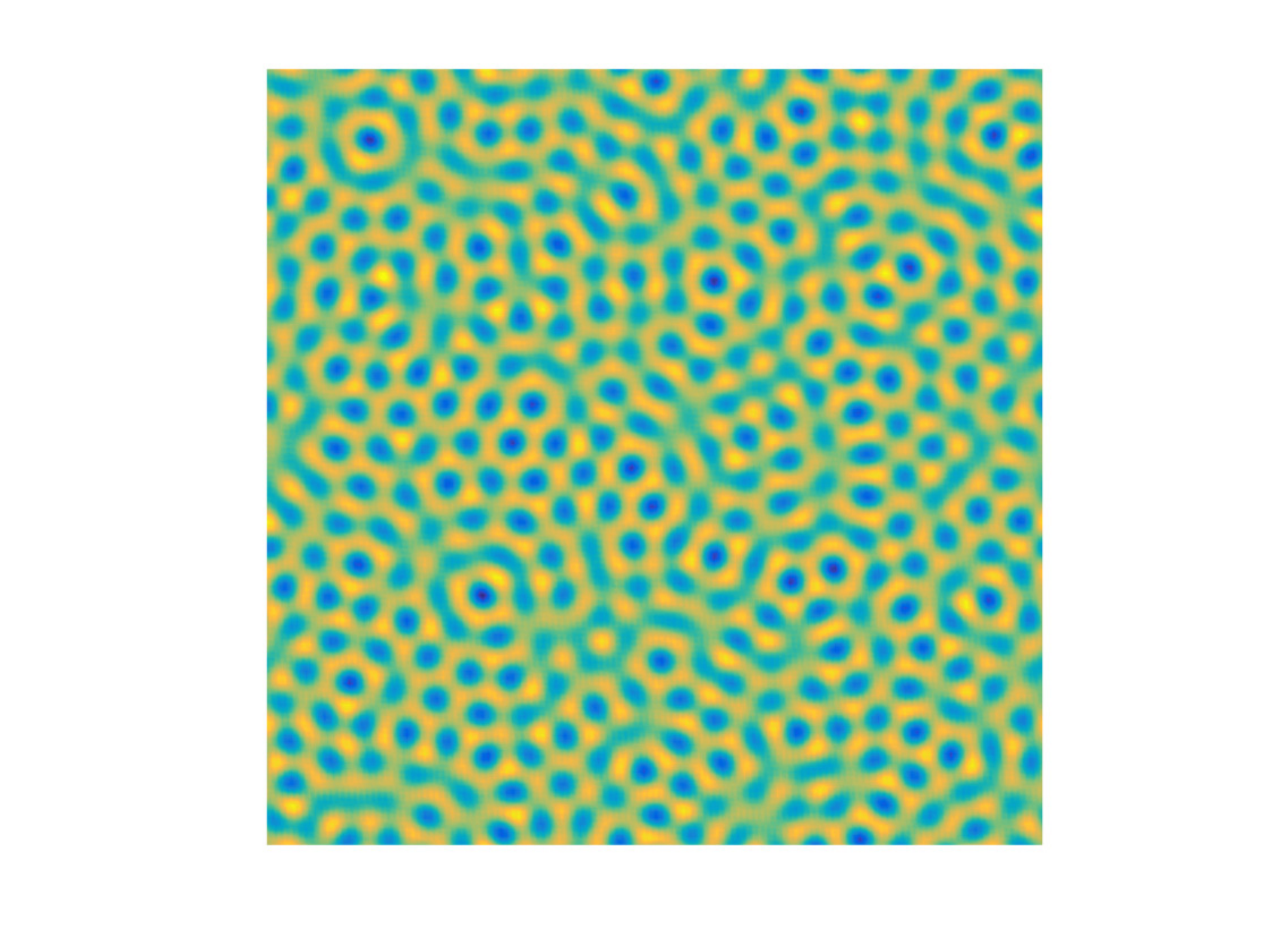}
\includegraphics[scale=0.21]{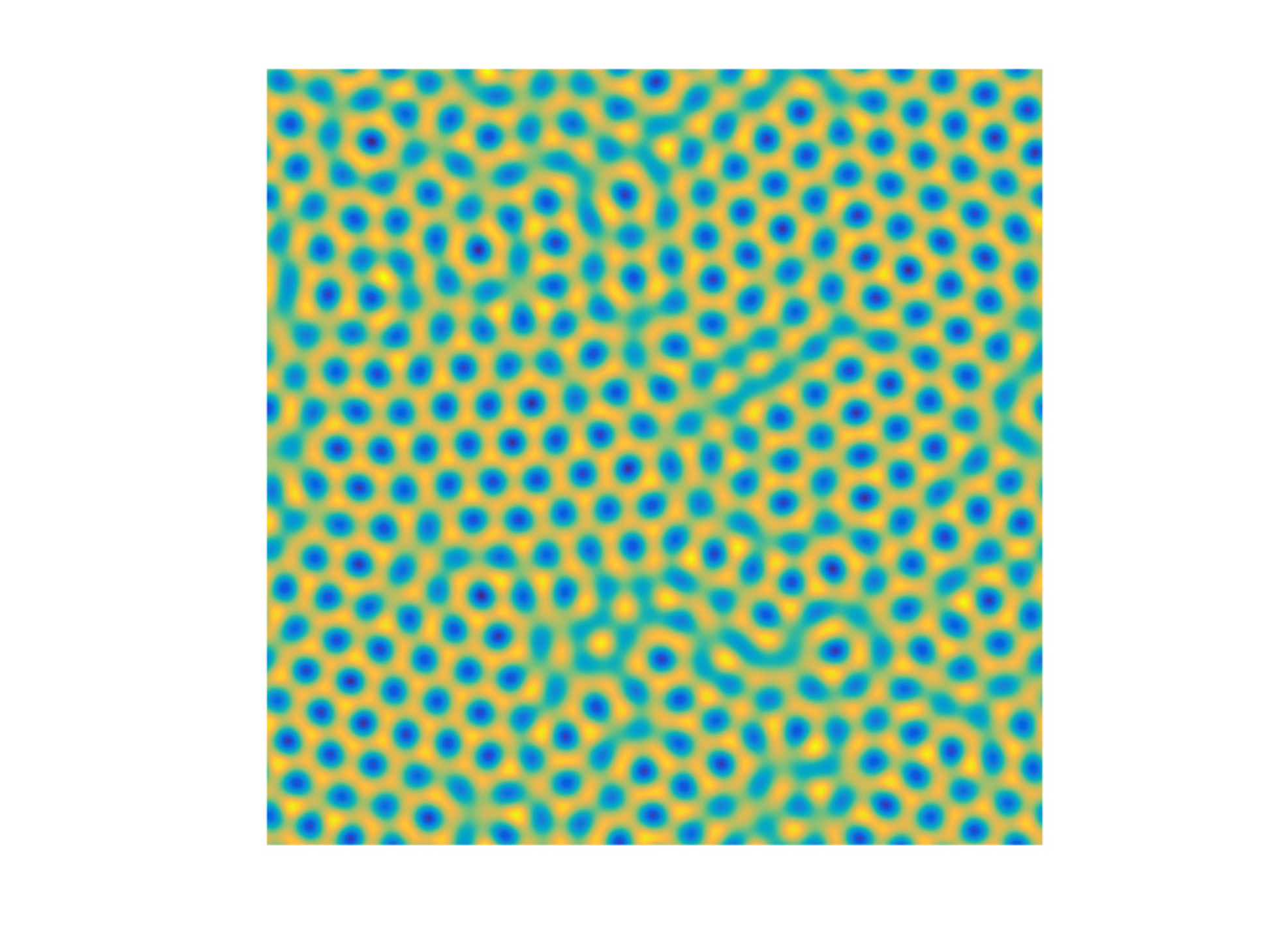}
\includegraphics[scale=0.21]{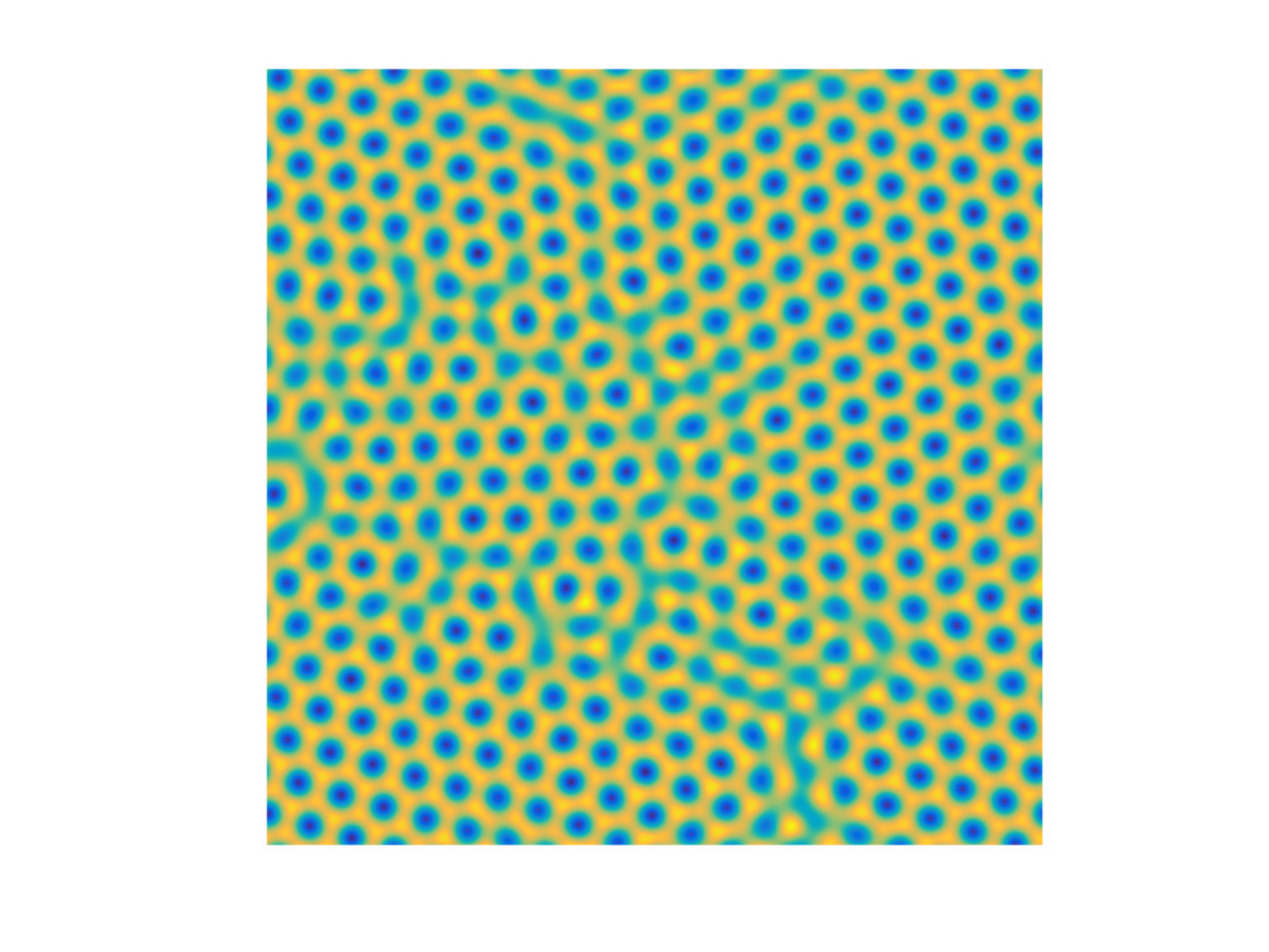}
\includegraphics[scale=0.21]{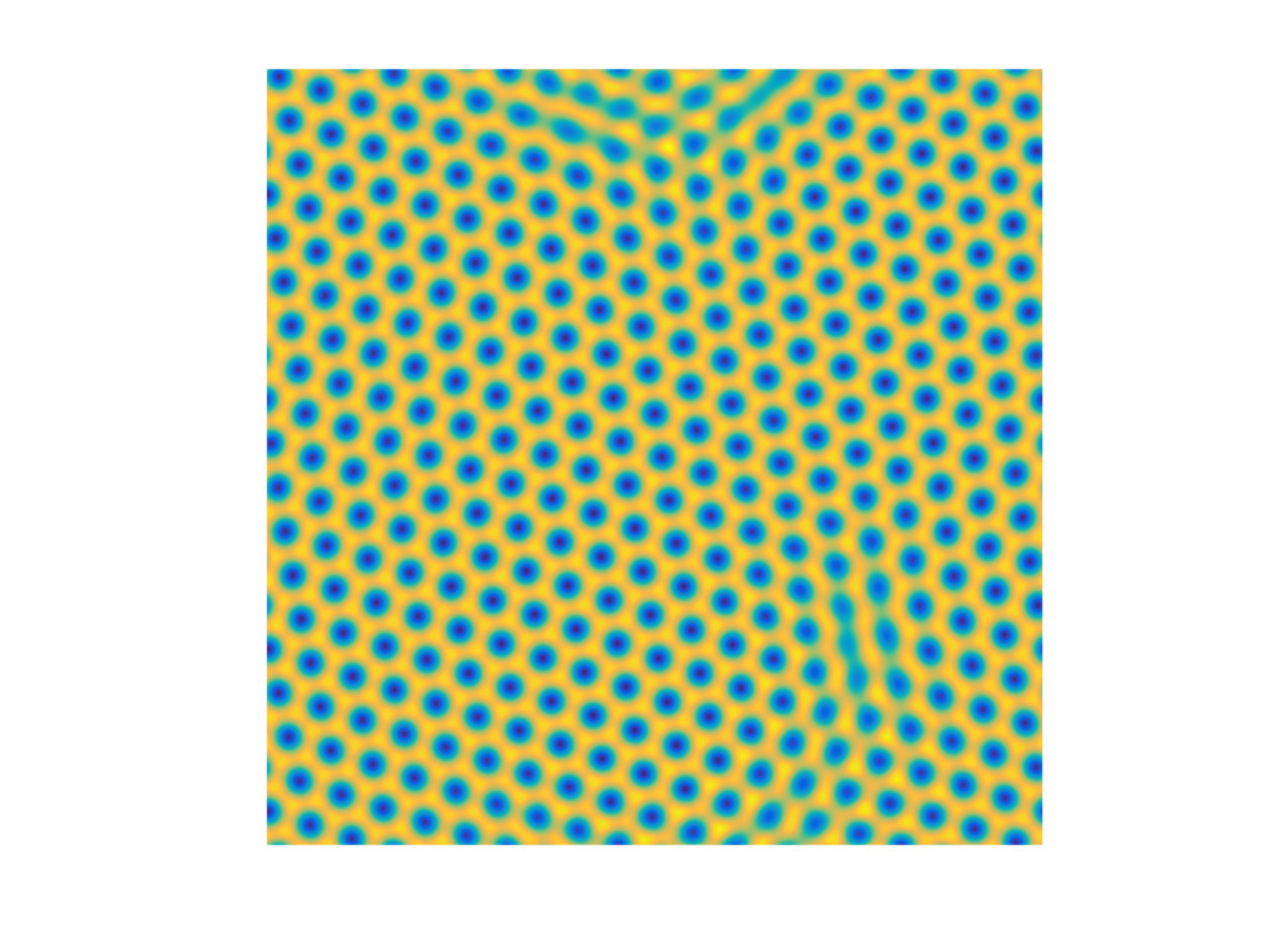}
\caption{The evolution of the density field $\phi$ calculated using the second order scheme at $t=150$, $260$, $400$, $500$, $1000$, $2000$, respectively.}  \label{fig: long time simulation}
\end{figure}

\subsection{Crystal growth in a supercooled liquid}
In this subsection, we simulate the crystal growth in a supercooled liquid with the following 
expression to define the crystallites:
\begin{equation}\label{e_initial_crystal growth}
\aligned
\phi_0(x_l,y_l)=\phi_{ave}+C_1\left(cos(\frac{C_2}{\sqrt{3}}y_l)cos(C_2 x_l)-0.5cos(\frac{2C_2}{\sqrt{3}}y_l)\right), \ \ l=1,2,3,
\endaligned
\end{equation}
where $x_l$ and $y_l$ define a local system of cartesian coordinates that is oriented with the crystallite lattice, and the constant parameters $\phi_{ave}$, $C_1$ and $C_2$ take the values $\phi_{ave}=0.285$, $C_1=0.446$ and $C_2=0.66$. Then we define the initial 
configuration by setting three perfect crystallites in three small square patches which are located at $(350,400)$, $(200,200)$ and $(600,300)$ with the length of each square is 40, similar numerical examples can be found in \cite{li2019efficient}. To generate crystallites with different orientations, we use the following affine transformation to produce a rotation given by three different angles $\theta=-\frac{\pi}{4},0,\frac{\pi}{4}$ respectively:
\begin{equation}\label{e_affine transformation}
\aligned
x_l(x,y)=x\sin(\theta)+y\cos(\theta), \ \ y_l(x,y)=-x\cos(\theta)+y\sin(\theta).
\endaligned
\end{equation}

In this simulation, we take the parameters $\epsilon=0.25$, $M=1$, $S=0.1$, $\lambda=0.001$ and $T=800$. $512^2$ Fourier modes are used to discretize the space and $\Delta t=0.05$. Figure \ref{fig: crystal growth} demonstrates the evolution of the phase transition behavior using the second-order scheme \eqref{e_fully second order discrete} at different times $t=0$, 50, 100, 200, 300, 400, 500, 600, 800, respectively. One can  observe  the growth of the crystalline phase and the motion of well-defined crystal-liquid interfaces. Besides, we can see that the different alignment of the crystallites causes defects and dislocations.

\begin{figure}[!htp]
\centering
\includegraphics[scale=0.21]{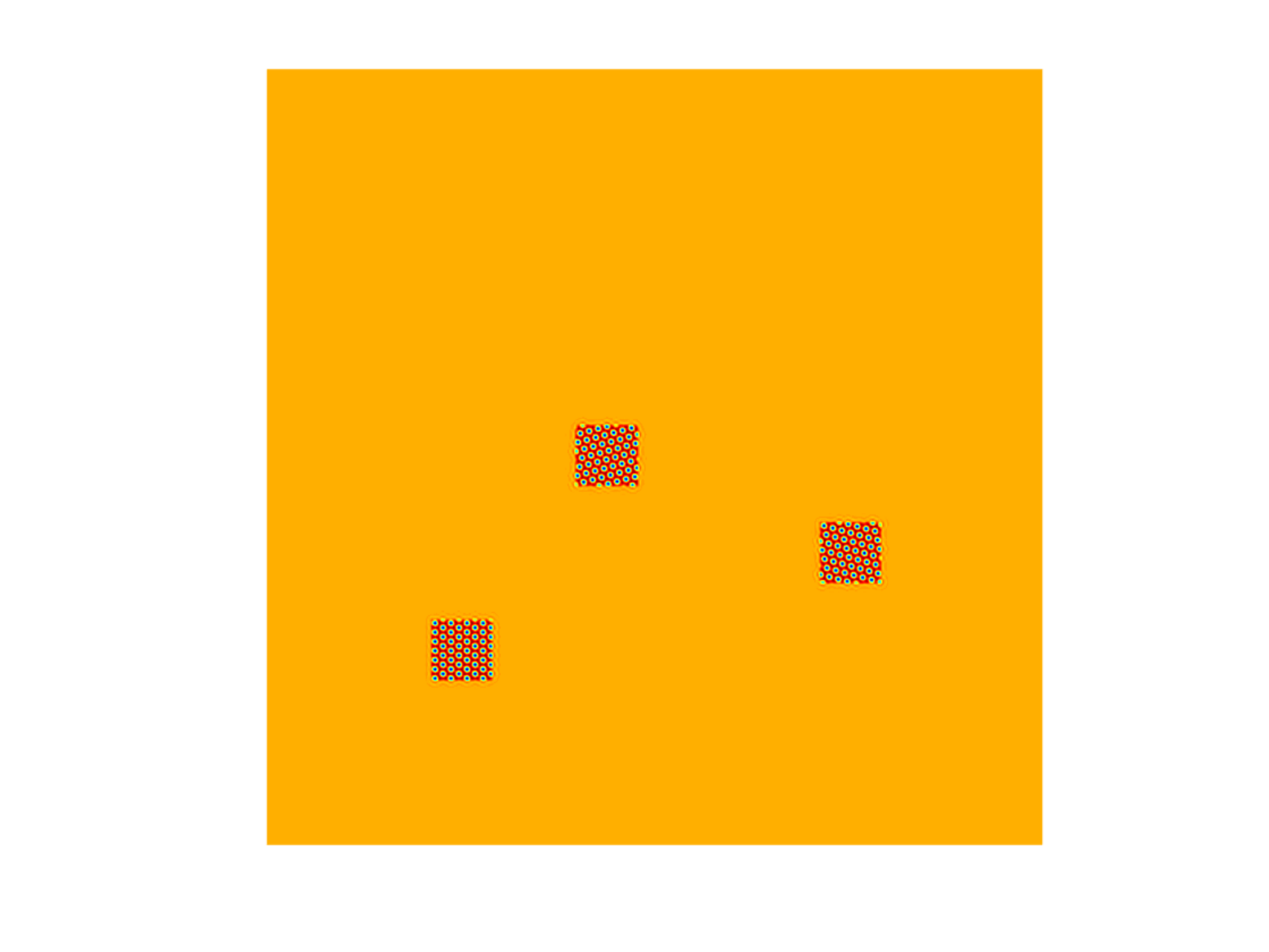}
\includegraphics[scale=0.21]{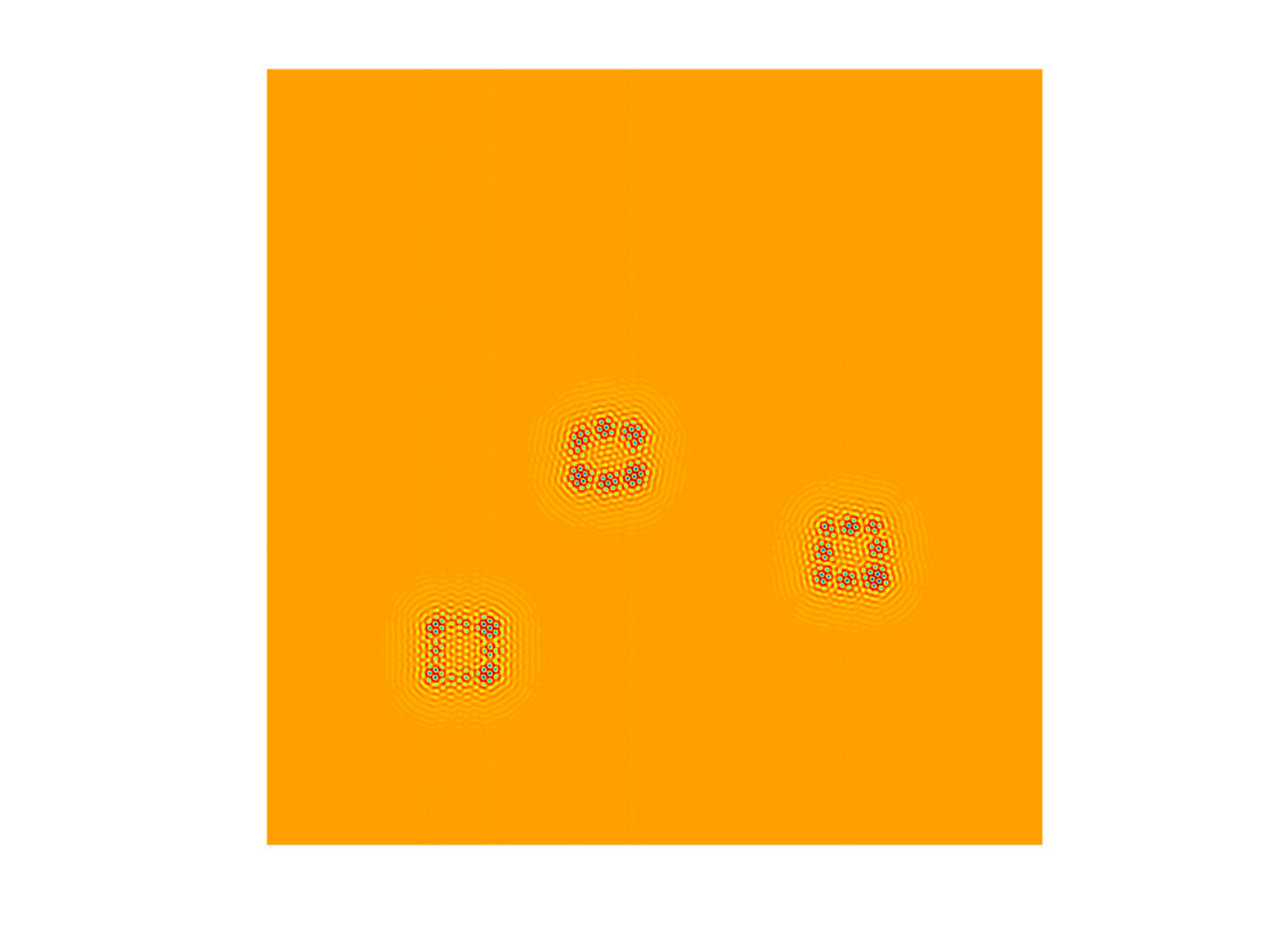}
\includegraphics[scale=0.21]{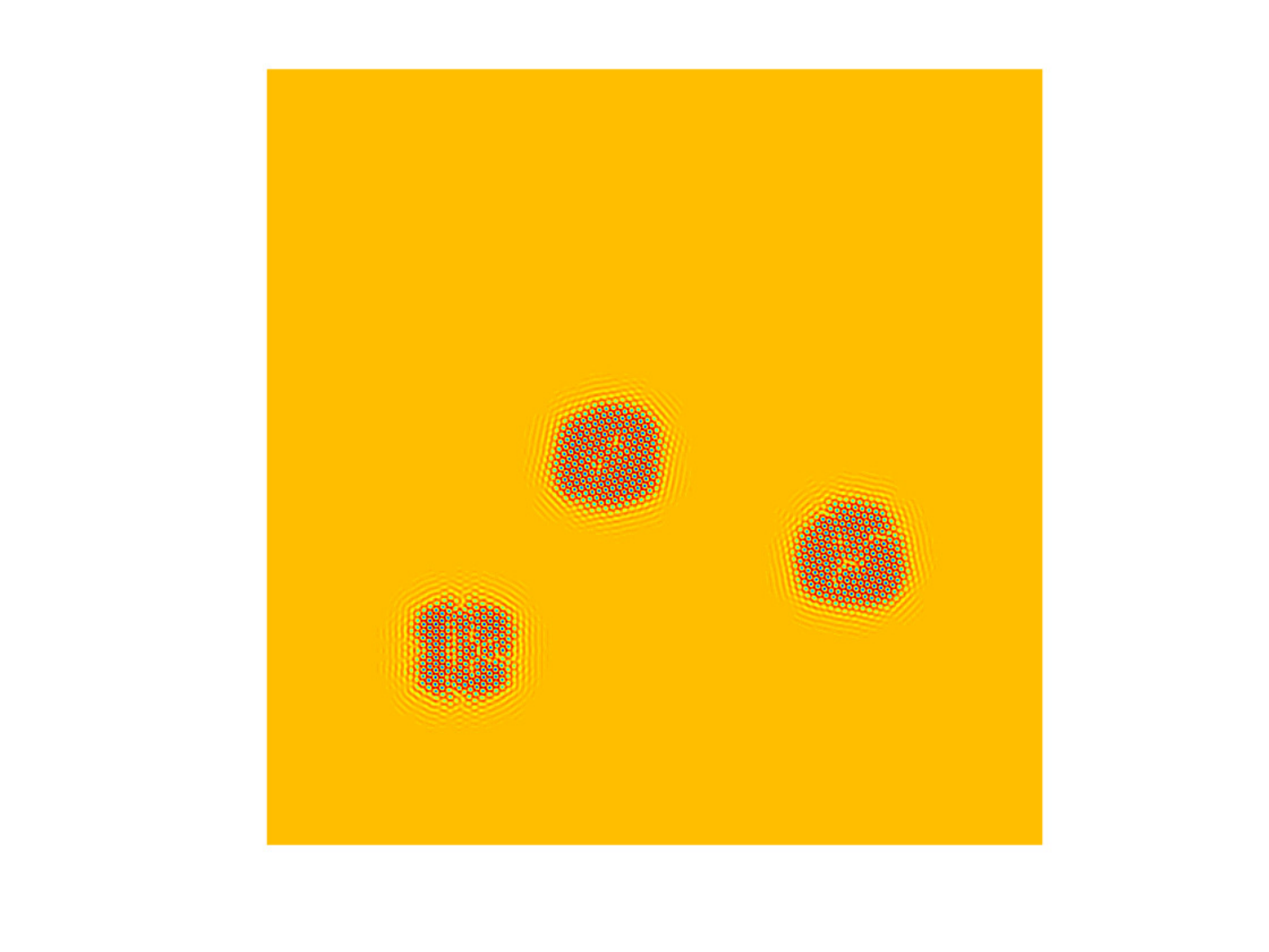}
\includegraphics[scale=0.21]{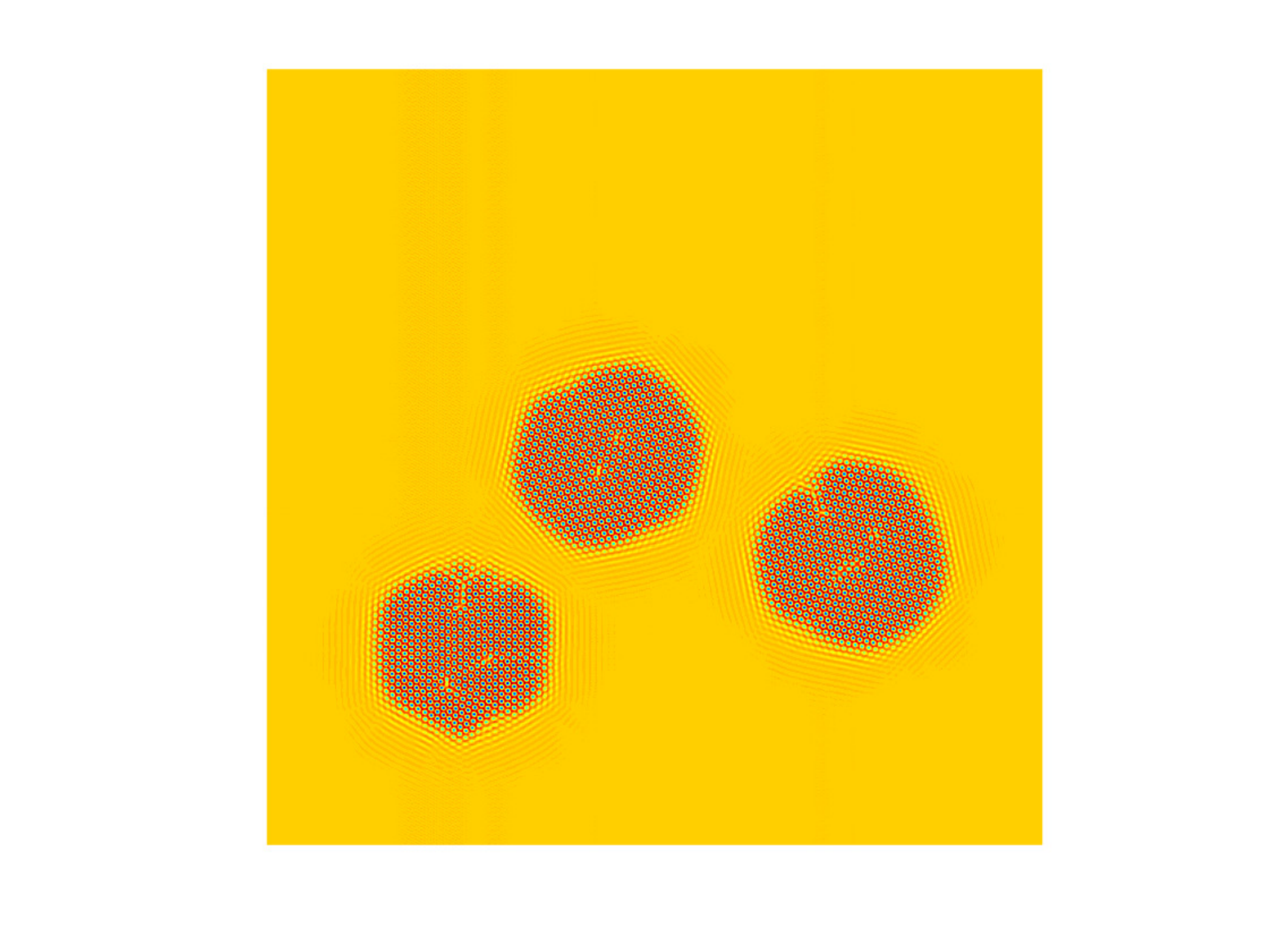}
\includegraphics[scale=0.21]{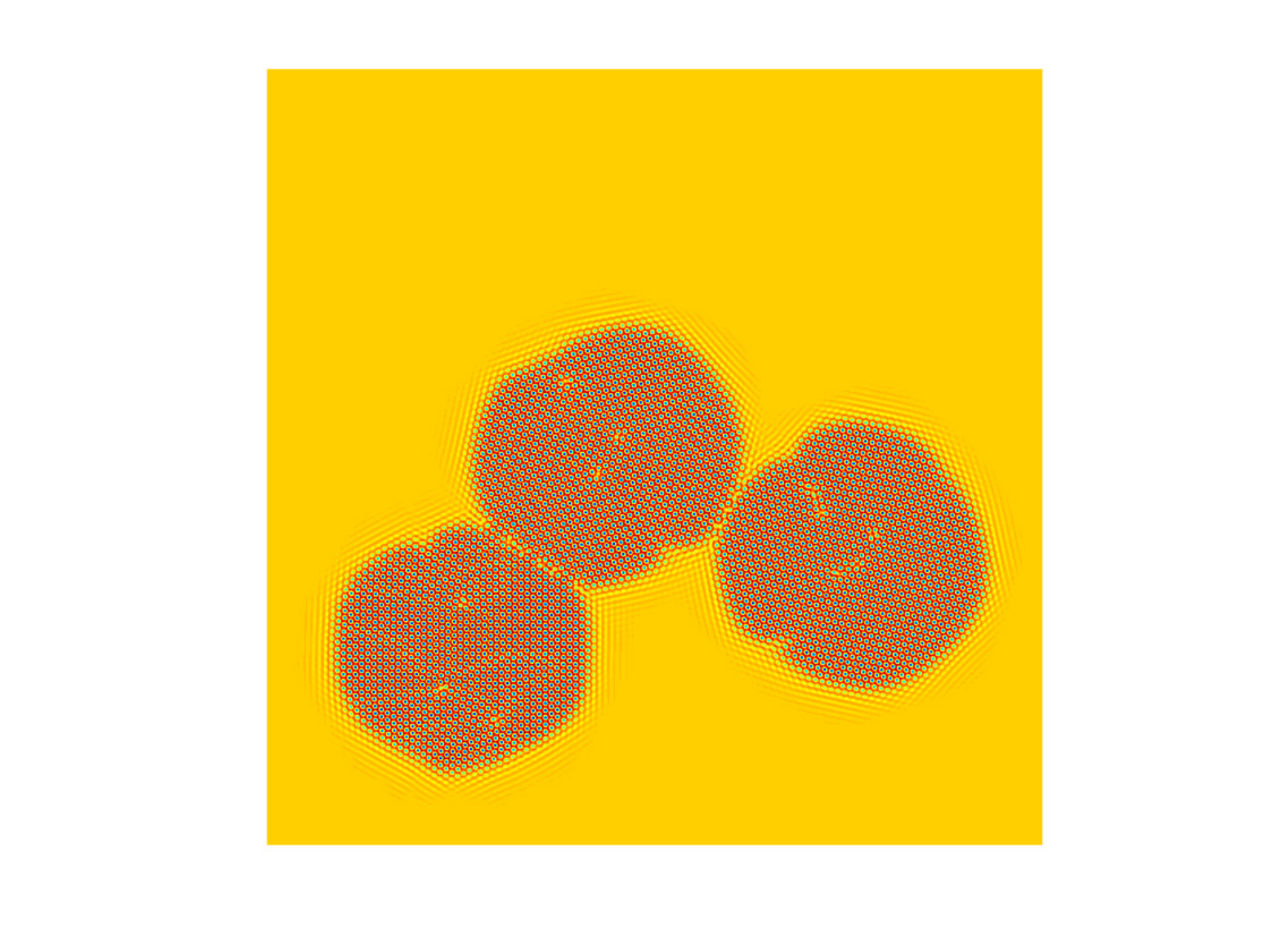}
\includegraphics[scale=0.21]{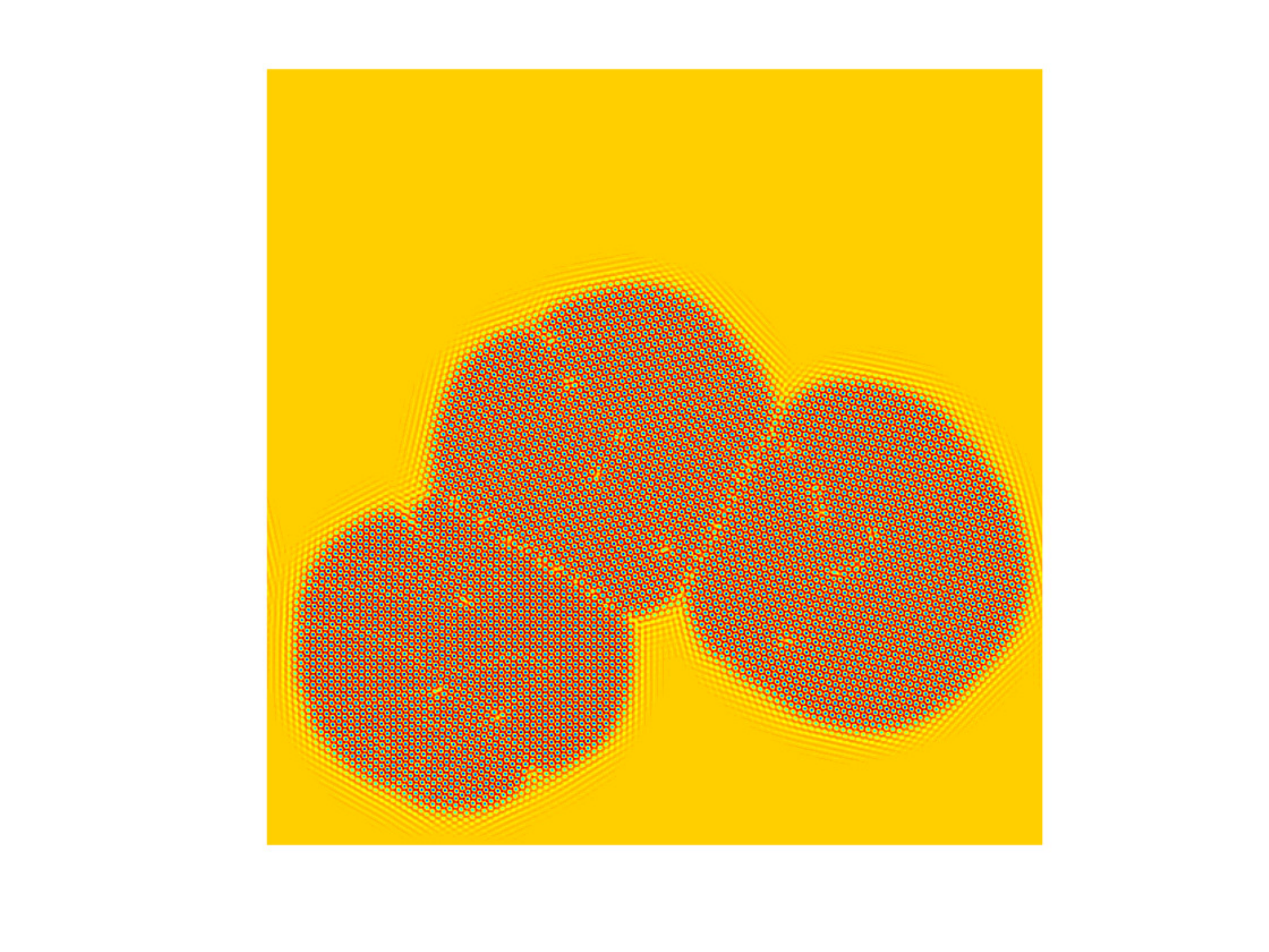}
\includegraphics[scale=0.21]{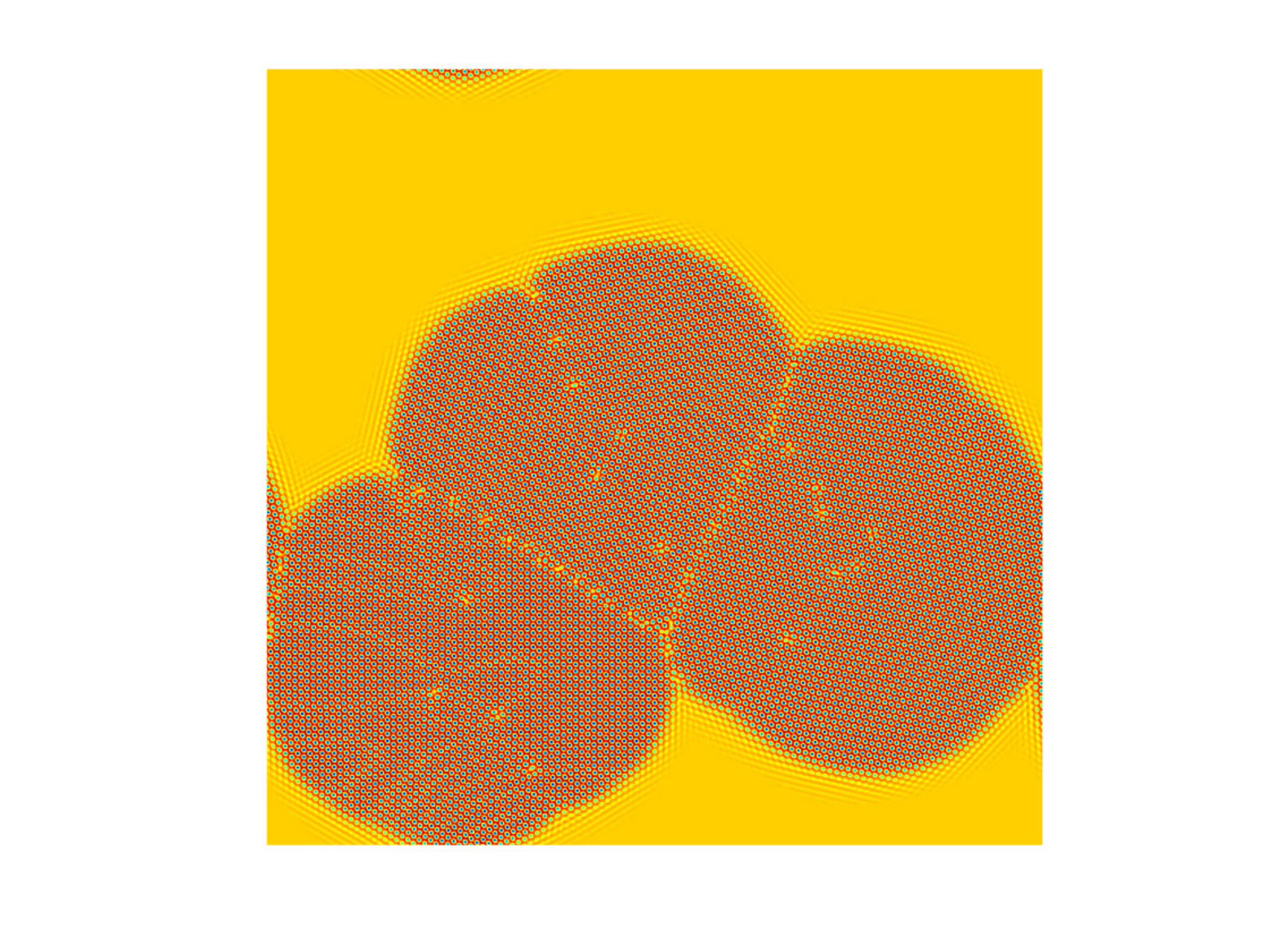}
\includegraphics[scale=0.21]{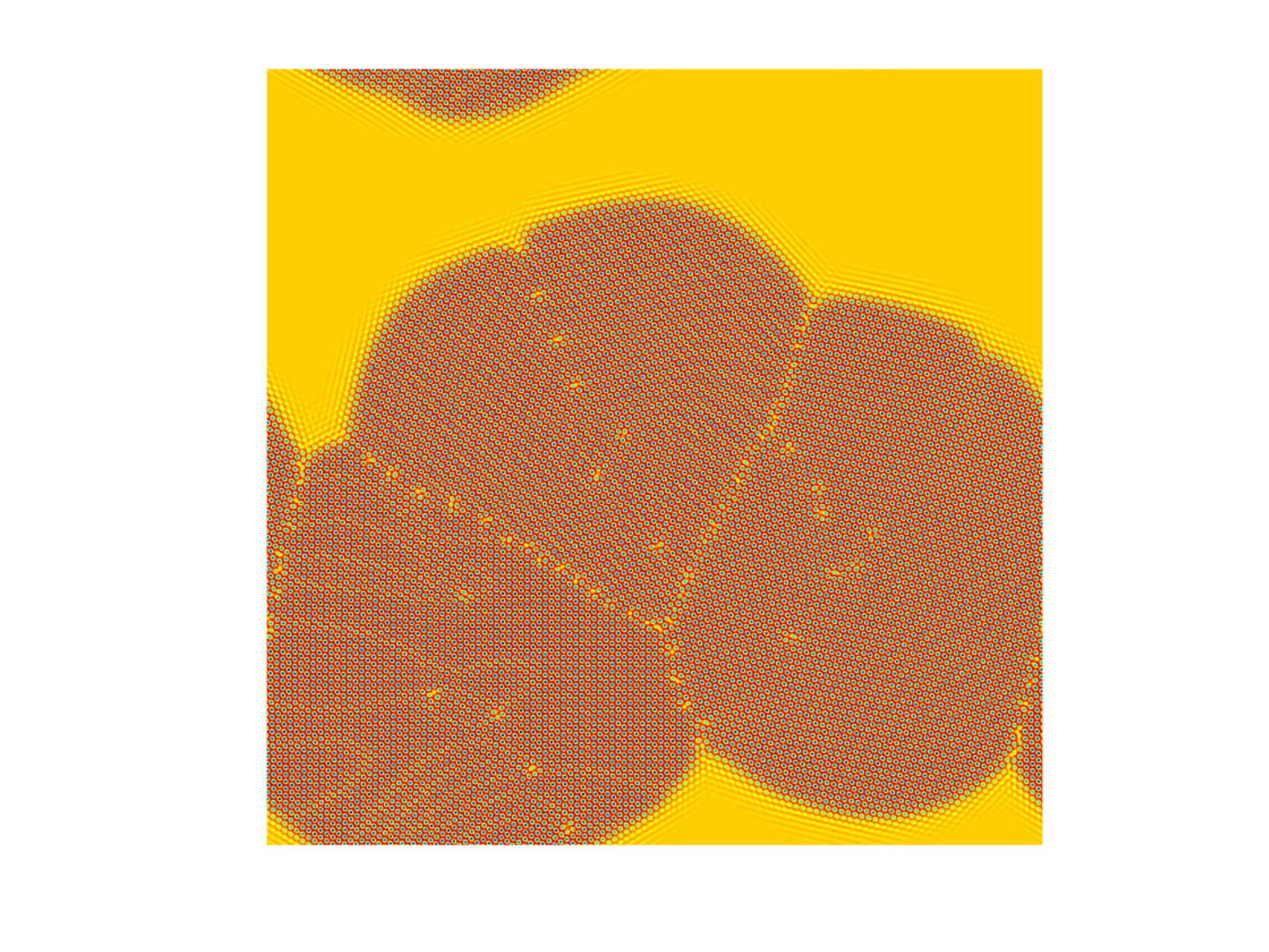}
\includegraphics[scale=0.21]{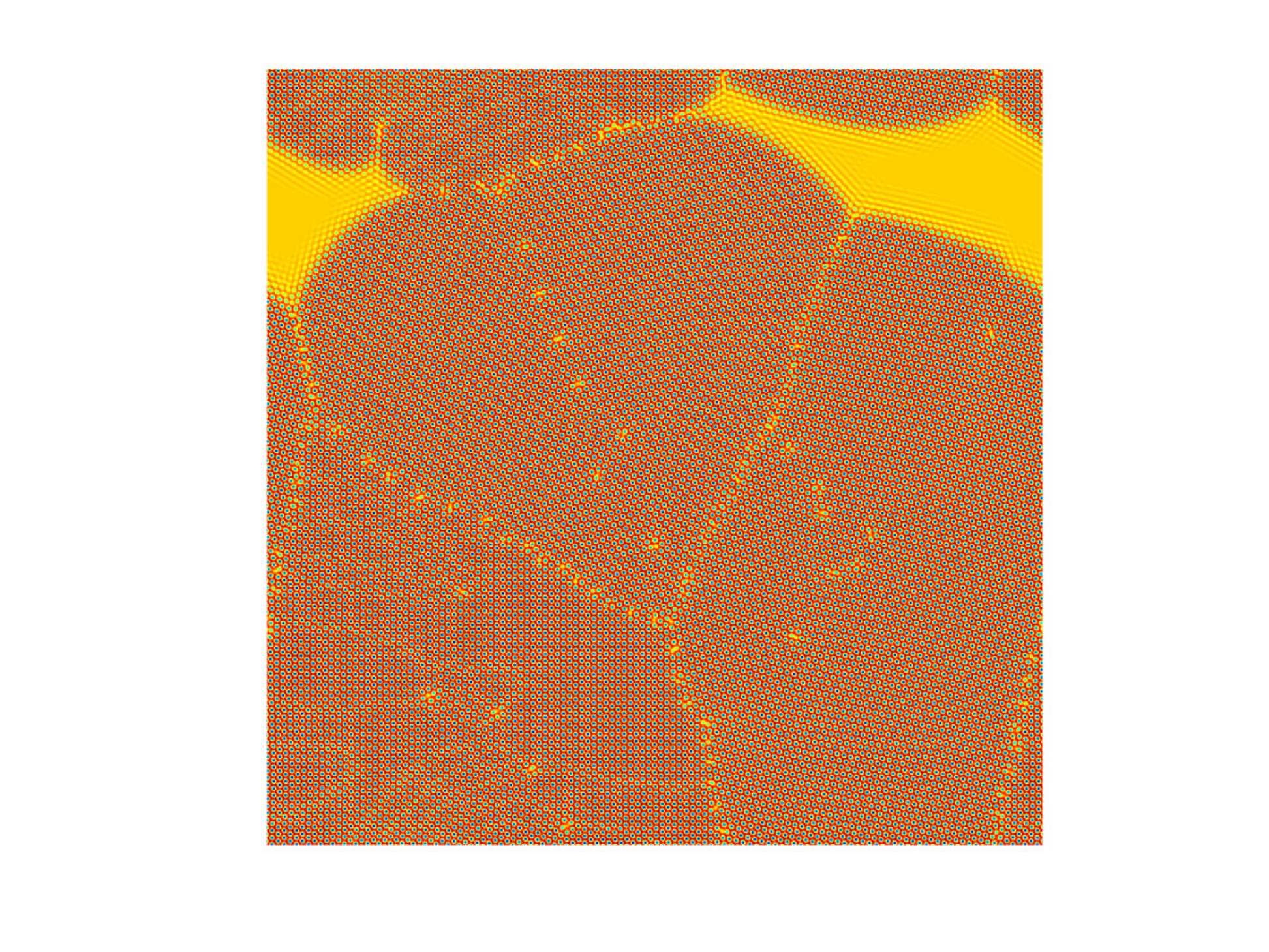}
\caption{The dynamical behaviors of the crystal growth in a supercooled liquid. Snapshots of the numerical approximation of the density field $\phi$ are taken at $t=0$, 50, 100, 200, 300, 400, 500, 600, 800, respectively.}  \label{fig: crystal growth}
\end{figure}
\section{Conclusion}

We developed  fully discrete, unconditionally energy stable schemes based on the  SAV  and stabilized SAV approaches in time and the Fourier-spectral method in space for the phase field crystal (PFC) equation.   We also carried out a rigorous error analysis which provided optimal error estimates in both time and space.  To the best of our knowledge, this is the first such result for any fully discrete linear schemes for the PFC model.

We showed that while the stabilization term is not required for stability or convergence, it is essential for the schemes to achieve reasonable accuracy without using exceedingly small time steps. 
We presented  numerical experiments  to demonstrate the accuracy and robustness of our schemes for the PFC model.


\bibliographystyle{siamplain}
\bibliography{Spectral_PFC}

\begin{thebibliography}{10}

\bibitem{ainsworth2017analysis}
{\sc M.~Ainsworth and Z.~Mao}, {\em Analysis and approximation of a fractional
  {C}ahn-{H}illiard equation}, SIAM Journal on Numerical Analysis, 55 (2017),
  pp.~1689--1718.

\bibitem{elder2004modeling}
{\sc K.~Elder and M.~Grant}, {\em Modeling elastic and plastic deformations in
  nonequilibrium processing using phase field crystals}, Physical Review E, 70
  (2004), p.~051605.

\bibitem{elder2002modeling}
{\sc K.~Elder, M.~Katakowski, M.~Haataja, and M.~Grant}, {\em Modeling
  elasticity in crystal growth}, Physical review letters, 88 (2002), p.~245701.

\bibitem{gomez2012unconditionally}
{\sc H.~Gomez and X.~Nogueira}, {\em An unconditionally energy-stable method
  for the phase field crystal equation}, Computer Methods in Applied Mechanics
  and Engineering, 249 (2012), pp.~52--61.

\bibitem{guo2016local}
{\sc R.~Guo and Y.~Xu}, {\em Local discontinuous {G}alerkin method and high
  order semi-implicit scheme for the phase field crystal equation}, SIAM
  Journal on Scientific Computing, 38 (2016), pp.~A105--A127.

\bibitem{li2019efficient}
{\sc Q.~Li, L.~Mei, X.~Yang, and Y.~Li}, {\em Efficient numerical schemes with
  unconditional energy stabilities for the modified phase field crystal
  equation}, Advances in Computational Mathematics,  (2019), pp.~1--30.

\bibitem{li2019energy}
{\sc X.~Li, J.~Shen, and H.~Rui}, {\em Energy stability and convergence of
  {SAV} block-centered finite difference method for gradient flows},
  Mathematics of Computation,  (2019).

\bibitem{li2017efficient}
{\sc Y.~Li and J.~Kim}, {\em An efficient and stable compact fourth-order
  finite difference scheme for the phase field crystal equation}, Computer
  Methods in Applied Mechanics and Engineering, 319 (2017), pp.~194--216.

\bibitem{ramos1991c}
{\sc J.~Ramos}, {\em C. canuto, my hussaini, a. quarteroni, ta zang, spectral
  methods in fluid dynamics, springer-verlag, new york (1988), dm 162}, 1991.

\bibitem{shen2011spectral}
{\sc J.~Shen, T.~Tang, and L.-L. Wang}, {\em Spectral methods: algorithms,
  analysis and applications}, vol.~41, Springer Science \& Business Media,
  2011.

\bibitem{shen2018convergence}
{\sc J.~Shen and J.~Xu}, {\em Convergence and error analysis for the scalar
  auxiliary variable ({SAV}) schemes to gradient flows}, SIAM Journal on
  Numerical Analysis, 56 (2018), pp.~2895--2912.

\bibitem{shen2018scalar}
{\sc J.~Shen, J.~Xu, and J.~Yang}, {\em The scalar auxiliary variable ({SAV})
  approach for gradient flows}, Journal of Computational Physics, 353 (2018),
  pp.~407--416.

\bibitem{shen2010numerical}
{\sc J.~Shen and X.~Yang}, {\em Numerical approximations of {A}llen-{C}ahn and
  {C}ahn-{H}illiard equations}, Discrete Contin. Dyn. Syst, 28 (2010),
  pp.~1669--1691.

\bibitem{swift1977hydrodynamic}
{\sc J.~Swift and P.~C. Hohenberg}, {\em Hydrodynamic fluctuations at the
  convective instability}, Physical Review A, 15 (1977), p.~319.

\bibitem{wise2009energy}
{\sc S.~M. Wise, C.~Wang, and J.~S. Lowengrub}, {\em An energy-stable and
  convergent finite-difference scheme for the phase field crystal equation},
  SIAM Journal on Numerical Analysis, 47 (2009), pp.~2269--2288.

\bibitem{yang2018efficient_schemes}
{\sc X.~Yang}, {\em Efficient schemes with unconditionally energy stability for
  the anisotropic {C}ahn-{H}illiard equation using the stabilized-{S}calar
  {A}ugmented {V}ariable ({S-SAV}) approach}, arXiv preprint arXiv:1804.02619,
  (2018).

\bibitem{yang2017linearly}
{\sc X.~Yang and D.~Han}, {\em Linearly first-and second-order, unconditionally
  energy stable schemes for the phase field crystal model}, Journal of
  Computational Physics, 330 (2017), pp.~1116--1134.

\bibitem{zhang2013adaptive}
{\sc Z.~Zhang, Y.~Ma, and Z.~Qiao}, {\em An adaptive time-stepping strategy for
  solving the phase field crystal model}, Journal of Computational Physics, 249
  (2013), pp.~204--215.

\end{thebibliography}

\end{document}